\documentclass[12pt,a4paper,reqno]{amsart}
\numberwithin{equation}{subsection}
\usepackage[utf8]{inputenc}
\usepackage[english]{babel}
\usepackage{hyperref}
\usepackage{xcolor}
\usepackage{enumitem}

\usepackage{multirow} 
\usepackage{graphicx}
\usepackage{array}
\usepackage{longtable}
\usepackage{tikz-cd}
\newtheorem{theo}{Theorem}[subsection]
\newtheorem{lem}{Lemma}[subsection]

\newtheorem{rem}{Remark}[subsection]
\newtheorem{prop}{Proposition}[subsection]
\newtheorem{example}{Example}[subsection]

\usepackage[foot]{amsaddr}
\usepackage{url}
\usepackage{hyperref}

\usepackage{amsmath}
\usepackage{amssymb, amscd}
\usepackage[all, cmtip]{xy}
\usepackage{xcolor}
\usepackage{multirow} 
\usepackage{longtable}
\usepackage{array}

{\sc}% Thm head font  (could be also \sc)
{}%         Body font
{}%         Indent amount (empty = no indent, \parindent = para indent)
{\sc}% Thm head font  (could be also \sc)
\newcommand{\thismonth}{\ifcase\month\or
  January\or February\or March\or April\or May\or June\or
  July\or August\or September\or October\or November\or December\fi
  \space\number\year}

\makeatletter
\newcommand{\rssymb}[2]{\newcommand{#1}{{\mathrmsl{#2}}}}
\newcommand{\calsymb}[2]{\newcommand{#1}{{\mathcal{#2}}}}
\newcommand{\bbsymb}[2]{\newcommand{#1}{{\mathbb{#2}}}}
\newcommand{\lieoper}[2]{\newcommand{#1}{\mathop
  {\mathfrak{#2}\null}\nolimits}}
\newcommand{\oper}[3][n]{\newcommand{#2}{\mathop
  {\mathrm{#3}\null}\ifx n#1\nolimits\else\limits\fi}}
\newcommand{\rsoper}[3][n]{\newcommand{#2}{\mathop
  {\mathrmsl{#3}\null}\ifx n#1\nolimits\else\limits\fi}}
\bbsymb\C{C} \bbsymb\F{F} \bbsymb\HQ{H}\bbsymb\N{N} \bbsymb\Q{Q}
\bbsymb\R{R} \bbsymb\U{U} \bbsymb\V{V} \bbsymb\W{W} \bbsymb\Z{Z}
\bbsymb\bbf{F} \bbsymb\bbk{K} \bbsymb\bbi{I} \bbsymb\bbl{L}
\bbsymb\bbo{O} \bbsymb\bbj{J} \bbsymb\bby{Y} \bbsymb\bbp{P}
\bbsymb\bba{A}
\calsymb\cA{A} \calsymb\cB{B} \calsymb\cC{C} 
\calsymb\cM{M} \calsymb\cN{N} \calsymb\cO{O} \calsymb\cP{P}
\calsymb\cU{U} \calsymb\cV{V} \calsymb\cW{W} \calsymb\cX{X}
\calsymb\cY{Y} \calsymb\cZ{Z}
\renewcommand{\geq}{\geqslant} \renewcommand{\leq}{\leqslant}
\oper\End{End} \oper\Hom{Hom}                    % Vector space constructions
\oper\Sym{Sym} \oper\Skew{Skew}
\oper\Aut{Aut}                                   % Group constructions
\oper\GL{GL} \oper\SL{SL}\oper\Symp{Sp} \oper\CO{CO} \oper\On{O}
\oper\SO{SO} \oper\Pin{Pin} \oper\Spin{Spin} \oper\CU{CU}
\oper\Un{U} \oper\SU{SU} \oper\PSU{PSU} \rsoper\Diff{Diff}
\rsoper\SDiff{SDiff}
\lieoper\der{der}                                % Lie algebra constructions
\lieoper\gl{gl} \lieoper\sgl{sl}\lieoper\symp{sp} \lieoper\co{co}
\lieoper\so{so} \lieoper\spin{spin} \lieoper\cu{cu} \lieoper\un{u}
\lieoper\su{su} \rsoper\Vect{Vect} \rsoper\Ham{Ham}
\def\la#1{\hbox to #1pc{\leftarrowfill}}
\def\ra#1{\hbox to #1pc{\rightarrowfill}}

\newcommand{\Norm}[2][]{\bigl|\mkern-3mu\bigr|#2\bigr|\mkern-3mu\bigr|
  _{\lower1pt\hbox{${}_{#1}$}}}

\rsoper\dimn{dim}                           % dimension
\rsoper\grad{grad}                          % gradient
\rsoper\kernel{ker}\rsoper\image{im}        % kernel and image
\rsoper\alt{alt}   \rsoper\sym{sym}         % alternating and symmetric part
\rsoper\Ad{Ad}     \rsoper\ad{ad}           % adjoint action or bundle
\rsoper\CoAd{CoAd} \rsoper\coad{coad}       % coadjoint action
\rsoper\trace{tr}  \rsoper\trfree{tf}       % trace and tracefree part
\rsoper\detm{det}                           % determinant
\rsoper\Vol{Vol}                            % volume
\rsoper\divg{div}                           % divergence
\rsoper\sign{sign}                          % sign function
\rssymb\iden{id}                            % identity
\rssymb\vol{vol}                            % volume element
\oper\Imag{Im}\oper\Real{Re}                % real and imaginary
\newcommand{\sd}{{\raise1pt\hbox{$\scriptscriptstyle +$}}}
\newcommand{\asd}{{\raise1pt\hbox{$\scriptscriptstyle -$}}}
\newcommand{\sdasd}{{\raise1pt\hbox{$\scriptscriptstyle\pm$}}}
\newcommand{\asdsd}{{\raise1pt\hbox{$\scriptscriptstyle\mp$}}}

\rsoper\scal{scal}
\def\kahl/{k\"ahler}
\def\Kahl/{K{\"a}hler}

\begin{document}

\title[The moduli of Sasaki-Einstein metrics and invertible polynomials]
{The local moduli of Sasaki-Einstein rational homology 7-spheres and invertible polynomials}

\author[J. Cuadros]{Jaime Cuadros Valle$^1$}
\author[J. Lope]{Joe Lope Vicente$^1$}

\address{$^1$Departamento de Ciencias, Secci\'on Matem\'aticas,
Pontificia Universidad Cat\'olica del Per\'u,
Av. Universitaria 1801, Lima 32, Per\'u}
\email{jcuadros@pucp.edu.pe}
\email{j.lope@pucp.edu.pe}

\date{\thismonth}

\begin{abstract} We study the local moduli space of Sasaki-Einstein metrics  on links of invertible polynomials defining rational homology 7-spheres.  All these polynomials are either of cycle type or are given as Thom Sebastiani sums of a cycle block and another atomic block. We found that for polynomials of cycle type, the  local moduli spaces of Sasaki-Einstein metrics  are zero dimensional. For the Thom-Sebastiani sums of an atomic block and a cycle polynomial, the dimensions of the local moduli  spaces of Sasaki-Einstein metrics are positive in general. Since all the links under study in this article remain Sasaki-Einstein rational homology 7-spheres under the  Berglund-H\"ubsch rule from classical mirror symmetry \cite{BH, CGL}, we are able to  find solutions for the problem associated to the moduli for the  Berglund-H\"ubsch transpose duals of this type of links. For the purpose 
of doing this, we give  specific description of the moduli spaces of complex structures on the weighted quasismooth hypersurfaces cut out by the corresponding invertible polynomials and, in particular, from this description, we can produce families of quasismooth weighted hypersurfaces  that denegerate to non-quasismooth with at worst klt singularties.

%In this setting,  non-quasismooth hypersurfaces play the role of points of the compactifycation of this moduli. Actually, 

%As Sasaki-Einstein geometry determines both a transverse K\"ahler-Einstein structure and a Calabi-Yau metric cone structure,  our  results  give explicit descriptions of the local  moduli of the complex structure for the corresponding K\"ahler-Einstein orbifold which in this setting is a  rational homology projective 3-space with quotient singularities and also give information on the moduli of Calabi-Yau cones associated to the Sasaki-Einstein 7-manifold. 

 \end{abstract}

\maketitle
% !TEX root =  % !TEX root =  

\medskip

\medskip
\tableofcontents

\maketitle
\vspace{-2mm} 

% !TEX root = MainMod.tex

\section{Introduction}

%The moduli problem, that is, the study of spaces that classify geometric structures on manifolds up to isomorphism, is highly relevant in Geometry. In many cases, finding a space with suitable properties  that parametrizes the isomorphism classes is too hard and  even if the space is formally described, sometimes in practice it is not possible to provide effective tools to compute the moduli space on explicit examples.  Sometimes the local problem  is  easier to handle and this  was the approach taken by Kodaira and Spencer in their canonical work  on deformations of complex structures on compact manifolds \cite{KSp}.   

For Sasakian manifolds, which are roughly speaking the odd dimensional analogue of K\"ahler manifolds, the moduli problem has been addressed and an appropriate notion of moduli space has been  designed, however little is known about this space, see  \cite{Bo} and references therein.   In particular,  finding the number of components of this moduli has been studied  in order  to obtain lower bounds for the dimension of the moduli space of Sasakian structures for links of Brieskorn-Pham polynomials and Smale manifolds  \cite{BMvK, Ko, BGK}. For Sasaki-Einstein structures, the local study on its moduli has given interesting results via deformations of the transverse complex structure in a more general setting, see \cite{No, vC}. 

In this article, we describe the local moduli of Sasaki-Einstein metrics from the study of invertible polynomials, that is  polynomials of the form $f=\sum_{i=1}^n \prod_{j=1}^n x_j^{a_{i j}},$
 where $A=\left(a_{i j}\right)_{i, j=1}^n$ is a non-negative integer-valued matrix which is invertible over $\mathbb{Q}$ and where 
 $f$ is quasihomogeneous, i.e., there exist positive integers $w_j$ such that $d:=\sum_{j=1}^n w_j a_{i j}$ is constant for all $i$, and
 $f$ is quasismooth, i.e., $f: \mathbb{C}\left[x_1, \ldots, x_n\right] \rightarrow \mathbb{C}$ has exactly one critical point at the origin. Due to the Kreuzer-Skarke classification of invertible polynomials \cite{KS}  we know that any invertible polynomial, up to permutation of variables, can be written as a Thom-Sebastiani sum of three types of polynomials usually called atoms:
\begin{enumerate}
\item Fermat type: $w=x^a$,
\item Chain type: $w=x_1^{a_1} x_2+x_2^{a_2} x_3+\ldots+x_{n-1}^{a_{n-1}} x_n+x_n^{a_n}$, and
\item Loop or cycle type: $w=x_1^{a_1} x_2+x_2^{a_2} x_3+\ldots+x_{n-1}^{a_{n-1}} x_n+x_n^{a_n} x_1$.
\end{enumerate}
These atoms have remarkable properties, for instance in \cite{Ko1} cyclic polynomials have been studied in order to obtain information on rational surfaces with quotient singularities. Invertible polynomials, thought as  defining  subsets  of weighted projective spaces,  produce smooth links with some interesting features both at the topological/differential and the Riemannian level.  As a matter of fact, in \cite{BG1} Boyer and  Galicki  generalized  a method introduced by Kobayashi in \cite{Kob} to produce Sasaki-Einstein metrics on links of hypersurface singularities.  From there,  many examples of highly connected  Sasaki-Einstein manifolds were found. For such examples  the existence of orbifold Kähler-Einstein metrics on various Fano orbifolds is verified using different incarnations  of the $\alpha$-invariant method developed by Tian, Nadel, Demailly,  Koll\'ar and others \cite{Ti,Na,DK, CS}.  In \cite{Ko, BGN1, BGK, BN} special  interest was given to 5-dimensional manifolds and homotopy spheres realized as links of Brieskorn-Pham polynomials which are an important block in the classification of invertible polynomials.  In these articles the local moduli space of Sasaki-Einstein structures  for this kind of links were also studied. For instance, in \cite{BGK} it was found that for homotopy spheres given as links, the dimension of their moduli grew double exponentially with the dimension and  they also showed that all the 28 oriented diffeomorphism classes on $S^7$ admitted  inequivalent families of Sasakian-Einstein structures. In \cite{BG2}, again using links of Brieskorn-Pham polynomials viewed as branched covers over Calabi-Yau hypersurfaces, it is shown that
there exist continuous parameter families of Sasakian-Einstein metrics on infinitely many simply connected rational homology spheres in every odd dimension greater than 3. In particular, for dimension 7 they found 38 effective real parameters. Later, in \cite{BvC} it is  shown that the moduli space of positive Sasaki classes with vanishing first Chern class for  manifolds of the form  $S^{2 n} \times S^{2 n+1}$  or for odd dimensional homotopy spheres among others, has a countably infinite number of components of dimension greater than one and these contain no extremal Sasaki metrics at all. Again this information is  produced through links of  Brieskorn-Pham polynomials. 

In \cite{BGN2}, based on the complete list of Johnson and Koll\'ar of anticanonically embedded Fano 3-folds  \cite{JK}, many rational homology 7-spheres admitting Sasaki-Einstein metrics were constructed using the Kobayashi-Boyer-Galicki method. The ideas in \cite{BGN2} were extended in \cite{CL} and \cite{CGL} and moreover it was shown that  all the invertible polynomials producing rational homology 7-spheres admitting Sasaki-Einstein metrics were polynomials of cycle type, or iterated Thom-Sebastiani 
sums of a cycle type involving three variables and another atomic type involving two variables and they cannot be written as Brieskorn-Pham polynomials. 
%It is natural to try to determine the family of inequivalent Sasaki-Einstein structures on each of these rational homology 7-spheres which determine different components of the moduli as the ones described by previous results \cite{BG2} where only Brieskorn-Pham polynomials or their perturbation were studied.  
In this article we investigate the space of deformations of Sasaki-Einstein metrics for this sort of links. We found that only the latter case admits non-equivalent Sasaki-Einstein metrics.  Actually we show that  
\begin{itemize}
\item For polynomials of cycle type producing rational homology 7-spheres, the corresponding local moduli space of Sasaki-Einstein metrics  has  dimension zero, that is,  links of this sort  do not admit inequivalent families of Sasaki-Einstein structures, see Subsection 3.1. 
\item  For Thom-Sebastiani sums of a cycle type  and another atomic type, the dimension of  local moduli  space of Sasaki-Einstein metrics  are given in terms of  the rational weights, introduced by Milnor and Orlik  in \cite{MO}, for the corresponding weighted hypersurface and, in general, this dimension is positive, see Subsection 3.2. 
\end{itemize}
Since all the links under study in this article remain Sasaki-Einstein rational homology 7-spheres under the  Berglund-H\"ubsch rule from classical mirror symmetry \cite{CGL}, we found that 
rational homology 7-spheres given as links coming from polynomials  of this type do not admit inequivalent families of Sasaki-Einstein structures with the exception of five elements, see Subsection 3.3. 
In particular our results give information on the connected components of the moduli of rational homology 7-spheres with specific differential structure which can be expressed as  link of isolated hypersurface singularities coming from the Johnson and Koll\'ar list  of Fano 3-folds of index 1. Our findings can be interpreted in two different settings:

\begin{itemize}
\item Seifert $S^1$-bundles are rational homology spheres if and only if the corresponding orbifolds  are rational homology complex projective spaces \cite{Ko2}, so our results describe  some components of  the moduli space of rational homology complex projective 3-spaces with quotient singularities.
\item  Sasaki-Einstein structures on the manifold determine  Ricci-flat K\"ahler cone metrics on the corresponding affine cone \cite{BBG},  so our results give information on the moduli of Calabi-Yau  cones.
\end{itemize}
\smallskip

In order to obtain these results, we find the generators of the infinitesimal deformations of the  complex structures for orbifolds arising from families of polynomials of the types mentioned above, and then   determine the moduli problem for K\"ahler-Einstein metrics on the corresponding hypersurfaces embedded in weighted projective spaces. In fact,  arithmetic conditions on the invertible polynomials are imposed so the corresponding  links are rational homology  7-spheres  admitting  Sasaki-Einstein metrics. Then we  we use these arithmetic constraints  to obtain all possible monomials 
$$z_0^{x_0} z_1^{x_1} z_2^{x_2} z_3^{x_3} z_4^{x_4},$$ of certain degree $d.$   

Due to our interest in smooth Sasaki-Einstein links, our results are based on the information obtained on the subset all quasismooth elements determined by the monomials in $H^{0}(\mathbb P(\mathbf{w}), \mathcal{O}(d)).$ Nonetheless, the explicit description of $H^{0}(\mathbb P(\mathbf{w}),\mathcal{O}(d))$ leads to the study  of the non-quasismooth polynomial generated by these  monomials where, as suggested in \cite{LX} (see also \cite{Od1}), the ones with at worst klt singularities are the candidates to  give rise to  non-smooth links whose metric cones can be considered as some sort of degenerating Calabi-Yau cones. Actually, non-quasismooth klt hypersurfaces  with $\alpha$-invariant greater than one describe points in the  boundary of a compactification of the moduli of quasismooth K\"ahler-Einstein weighted hypersurface and on the $K$-moduli of Fano cones as described in  \cite{Od} for instance. We explore these ideas in the last section of this article and  explain how to  produce families of quasismooth weighted hipersurfaces  that denegerate to non-quasismooth klt varieties through examples.   
\medskip

The paper is organized as follows. Section 2 reviews the background material relevant  for this paper. In Section 3 we relate the problem of finding the moduli number of Sasaki-Einstein 
rational homology 7-spheres to solving the Diophantine equations subject to certain arithmetic conditions extracted from  the building blocks of the invertible polynomial associated to this problem. Finally, we find the dimension of the corresponding  local moduli space of Sasaki-Einstein  metrics and prove our main results. In the last section, we discuss the role of non-quasismooth hypersurfaces in this setting.

\section{Preliminaries}
\subsection{Sasakian structures on smooth links and invertible polynomials} Below we give a very brief review of the main ingredients we need to establish our results.  
The canonical references here are  \cite{BBG} and \cite{Sp}. After that,  we give arithmetic conditions that allow us to  manufacture  rational homology 7-spheres from links of invertible polynomials. Finally,  we  explain how the Berglund-H\"ubsch transpose rule \cite{BH} can be used in this framework.
\medskip

\noindent{\bf Sasakian Geometry:} Sasakian geometry is a special type of contact metric structure on a $(2n+1)$-dimensional manifold $M$ described by the tensors $(\Phi, \xi, \eta, g)$  such that $\eta$ is a contact 1-form, $\Phi$ is an endomorphism of the tangent bundle, 
$g=d\eta \circ (\mathbb{I}\times \Phi)+\eta\otimes \eta$ is a Riemannian metric and $\xi$ is the Reeb vector field which is Killing. Moreover, the underlying CR-structure $\left(\mathcal{D},\left.\Phi\right|_{\mathcal{D}}\right)$ is integrable, where $\mathcal{D}=\ker \eta$ denotes the contact structure. 

A Sasaki manifold has a transverse  Kählerian structure  $(\nu(\mathcal{F}_\xi), \bar{J})$ 
determined by  the normal bundle $\nu(\mathcal{F}_\xi)$ of the characteristic foliation $\mathcal{F}_\xi$ determined by $\xi$ and the natural complex structure $\bar{J}$ in $\nu(\mathcal{F}_\xi)$ induced by $\left.\Phi\right|_{\mathcal{D}}.$ In fact,  when all the orbits of $\xi$ are closed, the Reeb vector field $\xi$ generates a locally free circle action whose quotient is a Kähler  orbifold. In this case the Sasakian structure is called quasiregular. When the action is free it is called  regular and the quotient is a Kähler manifold. In the irregular case when the orbits of the Reeb vector field  are not closed the local quotients are Kähler. 

Alternatively, one can understand a Sasakian structure on the manifold $M$ in terms of the metric cone $C(M)=M \times \mathbb{R}^{+}$ with symplectic form $\omega=$ $d\left(r^2 \eta\right)$ where $r$ is the radial coordinate. Indeed, using the Liouville vector field $\Psi=r \partial_r$ we define a natural complex structure $I$ on $C(M)$ by
$$
I X=\Phi X+\eta(X) \Psi, \quad I \Psi=-\xi, 
$$
where $X$ is a vector field on $M$, and $\xi$ is understood to be lifted to $C(M)$. By adding the apex of the cone we obtain an affine variety $C(M) \cup\{0\}$ which has been intensely studied to obtain K-stability theorems in the manner of the work  of Chen-Donaldson-Sun in the context of Sasakian manifolds \cite{CSz}. We know that $(M, g)$ is Sasakian if the cone $(C(M),\omega, I, \bar{g} )$ is K\"ahler with K\"ahler metric given by $\bar{g}=d r^2+r^2 g.$ 
Moreover, the Sasakian structure corresponds to a polarized affine variety $(C(M) \cup\{0\}, I, \xi)$ polarized by the Reeb vector field $\xi$ and this affine variety admits a Ricci-flat K\"ahler cone metric if and only if $(M, g)$ admits a Sasaki-Einstein metric.
 
Let us  denote by $\mathcal{S}(M)$ the space of all Sasakian structures on $M$  with the $C^{\infty}$ Fréchet topology as sections of vector bundles, and let  $\mathcal{S}(M, \xi, \bar{J})$ be the subset of Sasakian structures with Reeb vector field $\xi$ and  transverse holomorphic structure $\bar{J},$ which is endowed with the subspace topology.  Consider $\mathcal{S E}(M)$ the subspace of Sasaki-Einstein metrics in $\mathcal{S}(M, \xi, \bar{J})$ and 
let 
$\mathfrak{Aut}(\bar{J})$ be  the group of complex automorphisms of $(C(M), I)$ that commute with $\Psi-i \xi$. This group descends to an action on $(M, \mathcal{S})$ commuting with $\xi.$   Then one defines the local moduli space of Sasaki-Einstein metrics \cite{Bo} by
$$
\mathfrak{M}^{S E}(M)=\mathcal{S} \mathcal{E}(M) / \mathfrak{Aut}(\bar{J})_0,
$$
where $\mathfrak{A u t}(\bar{J})_0$ denotes  the connected component  of  $\mathfrak{Aut}(\bar{J}).$

As we will see in Subsection 2.2, this space has a very concrete description for links of weighted homogeneous hypersurfaces.
\medskip

\noindent{\bf Weighted homogeneous hypersurfaces:} Consider the weighted $\mathbb{C}^*$-action on the affine space $\mathbb{C}^{n+1},$ defined by
$$
\left(z_0, \ldots, z_n\right) \mapsto\left(\lambda^{w_0} z_0, \ldots, \lambda^{w_n} z_n\right)
$$
where $\mathbf{w}=\left(w_0, \ldots, w_n\right)$ is a sequence of positive integers.
Then we obtain the weighted projective space with a canonical orbifold structure,  defined as the quotient space $$\mathbb{P}(\mathbf{w})=\left(\mathbb{C}^{n+1}-\{\mathbf{0}\}\right) / \mathbb{C}^*.$$ The weighted projective space $\mathbb{P}(\mathbf{w})$ is said to be well-formed if the weighted $\mathbb{C}^*$-action on $\mathbb{C}^{n+1}$ has trivial stabilizers in codimension 1, that is, when $\operatorname{gcd}\left(w_0, \ldots, \widehat{w}_i, \ldots, w_{n}\right)=1$ for each $i.$ Here the hat symbol means delete that corresponding element.  As an algebraic variety, a weighted projective space can be defined as   $$\mathbb{P}\left(w_0, \ldots, w_{n}\right)=\operatorname{Proj}(S(\mathbf{w})),$$ 
where $S(\mathbf{w})=\mathbb{C}\left[x_0, \ldots, x_{n}\right]$ is the graded polynomial ring such that the weight of each $x_i$ equals $w_i.$

Let us recall  that a polynomial $f \in \mathbb{C}\left[z_0, \ldots, z_n\right]$ is said to be a weighted homogeneous polynomial of degree $d$ and weight vector $\mathbf{w}=$ $\left(w_0, \ldots, w_n\right)$,  if for any $\lambda \in \mathbb{C}^{*}$
$$
f\left(\lambda^{w_0} z_0, \ldots, \lambda^{w_n} z_n\right)=\lambda^d f\left(z_0, \ldots, z_n\right).
$$
We  assume that $f$ is chosen so that the affine algebraic variety 
$$
V_f=\{f=0\} \subset \mathbb{C}^{n+1}
$$
is smooth everywhere except at the origin in $\mathbb{C}^{n+1}$ which is equivalent to say  that  
 the  weighted variety $$Z_f=(V_f-\{\mathbf{0}\})/ \mathbb{C}^* \subset \mathbb{P}(\mathbf{w})$$ is quasismooth of dimension $n.$ There are well-known conditions that determine when a specific polynomial determines a quasismooth weighted hypersurface \cite{F,JK}: 
 
 \begin{lem}
 A weighted hypersurface of degree $d$ in $\mathbb{P}\left(w_0, \ldots, w_4\right)$, where $d>w_i$, is quasismooth if and only if the following hold:
\begin{enumerate}
\item For each $i=0, \cdots, 4$ there is a $j$ and a monomial $z_i^{m_i} z_j$ of degree $d$. Here $j=i$ is possible.
\item For all distinct $i, j$ either there is a monomial $z_i^{b_i} z_j^{b_j}$ of degree $d$ or there exist monomials $z_i^{n_1} z_j^{m_1} z_k, z_i^{n_2} z_j^{m_2} z_l$ of degree $d$ with $\{k, l\} \neq\{i, j\}$ and $k \neq l$.
\item For every $i, j$ there exists a monomial of degree $d$ that does not involve either $z_i$ or $z_j$.
\end{enumerate}
\end{lem}
\smallskip

We have the following important result whose proof can be found in \cite{BBG}, Corollary 5.4.8. First, recall that the index $I$ of a weighted hypersurface is given by the difference $I=|\mathbf{w}|-d,$ where 
$|\mathbf{w}|=\sum_{i=0}^n w_i$  denotes the norm of the weight vector  $\mathbf{w}.$ The weighted  hypersurface is said to be Fano if $I>0.$

\begin{theo}
 Let $Z_f \subset \mathbb{P}(w_0, \ldots w_n)$ be a quasismooth weighted homogeneous Fano hypersurface of degree $d$. Then $Z_f$ admits a Kähler-Einstein orbifold metric if the following estimate holds:
$$
d I<\frac{n}{(n-1)} \min _{i, j}\left\{w_i w_j\right\}.
$$
\end{theo}
\medskip

The weighted hypersurface is well-formed if $\mathbb{P}(\mathbf{w})$ is well-formed and $Z_f \cap {\text {sing}}(\mathbb{P}(\mathbf{w}))$ has codimension at least 2 in $Z_f$. When $Z_f$ is well-formed, the canonical divisor satisfies the adjunction formula $K_{Z_f}=\mathcal{O}_{Z_{f}}\left(d-w_0-\cdots-w_{n}\right)$. In \cite{F} a criterion to determine the well-formedness of the weighted variety is given: a hypersurface defined by the weighted homogeneous polynomial $f$ of degree $d$ is well-formed in the well-formed weighted projective space $\mathbb{P}\left(w_0, \ldots, w_n\right)$ if 
 $\operatorname{gcd}\left(w_0, \ldots, \hat{w}_i, \ldots, \hat{w}_j, \ldots, w_n\right) \mid d$
for distinct $i, j=0, \ldots, n$.  
%In particular, if $X_d=X_f$, where $f$ is a Brieskorn-Pham polynomial, then $X_d$ is well-formed if and only if $\mathbb{P}\left(w_0, \ldots, w_n\right)$ is well-formed.
\medskip

\noindent{\bf Links of hypersurface singularities and Sasakian structures.}   Sasakian structures can be manufactured on  links of  hypersurface singularities of weighted homogeneous polynomials and we explain how to do this below.

 A link $L_f(\mathbf{w}, d),$ or  $L_f$  for short, is defined  as the intersection  $V_f \cap S^{2n+1}$, where $S^{2n+1}$ is the $(2n+1)$-sphere in Euclidean space. By the Milnor fibration Theorem \cite{Mi}, $L_f(\mathbf{w}, d)$ is a closed ($n-2)$-connected manifold that bounds a parallelizable manifold with the homotopy type of a bouquet of $n$-spheres. Furthermore, $L_f(\mathbf{w}, d)$ admits a quasiregular Sasaki structure $\mathcal{S}=\left(\xi_{\mathbf{w}}, \eta_{\mathbf{w}}, \Phi_{\mathbf{w}}, g_{\mathbf{w}}\right)$  which is the restriction of the weighted Sasaki structure on the sphere $S^{2 n+1}$ with Reeb vector field $\xi_{\mathrm{w}}=\sum_{k=0}^{n} w_k \left ( y_k \partial_{{x}_k}-x_k\partial_{{y}_k} \right )$ and contact form $\eta_{\mathbf{w}}=\frac{\eta}{\sum_{i=0}^n w_i\left(\left(x_i\right)^2+\left(y_i\right)^2\right)},$ where $\eta$ denotes the standard contact 1-form on the sphere $S^{2n+1}.$
 If one considers the locally free $S^1$-action induced by the weighted $\mathbb{C}^{*}$ action on $V_f,$ the quotient space of the link  $L_{f}(\mathbf{w}, d)$ by this action is the weighted hypersurface $Z_{f},$ a K\"ahler orbifold. We have the following commutative diagram \cite{BBG}

\begin{equation*}
\begin{CD} 
 L_{f}(\mathbf{w}, d) @> {\qquad\qquad}>>  S^{2n+1}_{\bf w}\\
@VV{\pi}V  @VVV\\
Z_{f}  @> {\qquad\qquad}>>  {\mathbb P}({\bf w}),
\end{CD}
\end{equation*}
where $S_{\mathbf{w}}^{2 n+1}$ denotes the unit sphere with a weighted Sasakian structure, $\mathbb{P}(\mathbf{w})$ is a weighted projective space coming from the quotient of $S_{\mathbf{w}}^{2 n+1}$ by a weighted circle action generated from the weighted Sasakian structure. The top horizontal arrow is a Sasakian embedding, the bottom arrow is a K\"ahlerian embedding and  the vertical arrows are orbifold Riemannian submersions. 

It follows from the orbifold adjunction formula that the  link $L_f$ admits a positive Ricci curvature if the quotient orbifold $Z_{f}$ by the natural  $S^1$-action  is Fano.  

In \cite{Kob}, Kobayashi showed that  the link of a cone over a smooth projective variety $Z \subset \mathbb{P}^n$ carries a natural Einstein metric if and only if $Z$ is Fano and $Z$ carries a Kähler-Einstein metric. In \cite{BG1}, the authors generalized this result to weighted cones and furthermore gave an algorithm, the Kobayashi-Boyer-Galicki method, to obtain $(n-1)$-connected Sasaki-Einstein $(2n+1)$-manifolds from the existence of  orbifold Fano Kähler-Einstein hypersurfaces $Z_{f}$ in weighted projective $2n$-space $\mathbb{P}(\mathbf{w}).$  
\medskip 

Important topological information of the link can be obtained via the Alexander polynomial. Recall that the  Alexander  polynomial $\Delta_f(t)$  \cite{Mi} associated to a link $L_f$ of dimension $(2n-1)$ is the characteristic polynomial of the monodromy map $h_*: H_{n}(F, \mathbb{Z}) \rightarrow H_{n}(F, \mathbb{Z})$  which is induced by the circle action on the Milnor fibre $F$. Then 
$\Delta_f(t)=\operatorname{det}\left(t {\mathbb I}-h_*\right)$. 
Now both $F$ and its closure $\bar{F}$ are homotopy equivalent to a bouquet of $n$-spheres $S^n \vee \cdots \vee S^n,$ and the boundary of $\bar{F}$ is the link $L_f$, which is $(n-2)$-connected. The Betti numbers $b_{n-1}\left(L_f\right)=b_{n}\left(L_f\right)$ equal the number of factors of $(t-1)$ in $\Delta_f(t)$.  From the  Wang sequence of the Milnor fibration (see \cite{Or}) $$\quad 0 \longrightarrow H_n\left(L_f, \mathbb{Z}\right) \longrightarrow H_n(F, \mathbb{Z}) \stackrel{\mathbb{I}-h_*}{\longrightarrow} H_n(F, \mathbb{Z}) \longrightarrow H_{n-1}\left(L_f, \mathbb{Z}\right) \longrightarrow 0$$ one obtains that  $L_f$ is a $\mathbb{Q}$-homology sphere if and only if $\Delta_f(1) \neq 0$ and the order  of $H_{n-1}\left(L_f, \mathbb{Z}\right)$ equals $|\Delta_f(1)|.$ 
In the case that $f$ is a weighted homogeneous polynomial  there is an algorithm due to Milnor and Orlik \cite{MO} to compute the Alexander polynomial in terms of the degree and the weights:  associate to any monic polynomial $f$ with roots $\alpha_1, \ldots, \alpha_k \in \mathbb{C}^*$ its divisor
$$
\operatorname{div} f=\left\langle\alpha_1\right\rangle+\cdots+\left\langle\alpha_k\right\rangle
$$
as an element of the integral ring $\mathbb{Z}\left[\mathbb{C}^*\right].$ Let  $\Lambda_n=\operatorname{div}\left(t^n-1\right)$. 
Then the divisor of $\Delta_f(t)$ is given by
\begin{equation*}
\operatorname{div} \Delta_f=\prod_{i=0}^n\left(\frac{\Lambda_{u_i}}{v_i}-\Lambda_1\right),
\end{equation*}
where the $u_i's$  and $v_i's$ are given  terms of the degree $d$ of $f$ and the weight vector ${\bf w}=(w_0,\ldots w_n)$ by the equations 
\begin{equation*}
u_i=\frac{d}{\operatorname{gcd}\left(d, w_i\right)}, \quad v_i=\frac{w_i}{\operatorname{gcd}\left(d, w_i\right)}.
\end{equation*}

\noindent{\bf Invertible polynomials and rational homology spheres:} In \cite{CL,CGL} we used the Kobayashi-Boyer-Galicki method to  establish  the existence of Sasaki-Einstein metrics on links of invertible polynomials. 
From  \cite{BGN2} and \cite{CL} one notices that all the invertible polynomials  taken from the list of Johnson and Koll\'ar of anticanonically embedded Fano 3-folds \cite{JK}  producing  Sasaki-Einstein rational homology 7-spheres were polynomials of cycle type, chain type, or iterated Thom-Sebastiani sums of these types. These polynomials  can be described in terms of the following two sets of conditions 
\begin{itemize}
\item The weights and the degree satisfy $\operatorname{gcd}(d,w_{i})=1$ for all $i=0,\dots 4,$ which leads to singularities of cycle type.
\item The weights are subject to  $\left(w_0, w_1, w_2, w_3, w_4\right)=\left(m_3 v_0, m_3 v_1, m_2 v_2, m_2 v_3, m_2 v_4\right)$ 
with  $\gcd(m_{2},m_{3})=1$ and $m_{2}m_{3}=d,$  which leads to singularities that can described as iterated  Thom-Sebastiani sums of chain, cycle type and Fermat singularities. More precisely, the types of singularities obtained have the following form: 
\begin{align*}
    \mbox{Type I (Fermat-Cycle): } & z_{0}^{a_{0}}+z_{1}^{a_{1}}+z_{4}z_{2}^{a_{2}}+z_{2}z_{3}^{a_{3}}+z_{3}z_{4}^{a_{4}}\\
    \mbox{Type II (Chain-Cycle): } & z_{0}^{a_{0}}+z_{0}z_{1}^{a_{1}}+z_{4}z_{2}^{a_{2}}+z_{2}z_{3}^{a_{3}}+z_{3}z_{4}^{a_{4}}\\
    \mbox{Type III (Cycle-Cycle): } & z_{1}z_{0}^{a_{0}}+z_{0}z_{1}^{a_{1}}+z_{4}z_{2}^{a_{2}}+z_{2}z_{3}^{a_{3}}+z_{3}z_{4}^{a_{4}}.
\end{align*}

\end{itemize}

Polynomials with this set of conditions satisfy  the following statements whose proofs can be found in \cite{BGN2} and/or in the proof of Theorem 3.2 in \cite{CL}:

\begin{enumerate}
\item Consider links $L(\textbf{w}, d)$ of  weighted homogeneous polynomials $f$ of the first kind, that is, with $\operatorname{gcd}(d,w_{i})=1$ for all $i=0,\dots 4,$ then the Milnor number 
$m(L_f)$ for $L_f$ is given by  $$m(L_f) + 1 = d(b_{n-1}+1) \,\, \mathrm{ and } \,\,
H_{3}\left(L_f, \mathbb{Z}\right)_{t o r}=\mathbb{Z}_d.$$  In particular,   if $f$ is given as a polynomial of cycle type 
$$f=z_{4}z_{0}^{a_{0}}+z_{0}z_{1}^{a_{1}}+z_{1}z_{2}^{a_{2}}+z_{2}z_{3}^{a_{3}}+z_{3}z_{4}^{a_{4}}$$
of degree $d$ in the projective space $\mathbb{P}(\textbf{w})$, where  $\textbf{w}=(w_{0},w_{1},w_{2},w_{3},w_{4})$  then from the equations 
$$\begin{aligned}
& a_0 w_0+w_4=d, \quad a_1 w_1+w_0=d, \\
& \quad a_2 w_2+w_1=d, \quad a_3 w_3+w_2=d, \quad a_4 w_4+w_3=d 
\end{aligned}
$$
we obtain 
$$a_0 a_1 a_2 a_3 a_4=\left(\frac{d-w_4}{w_0}\right)\left(\frac{d-w_0}{w_1}\right)\left(\frac{d-w_1}{w_2}\right)\left(\frac{d-w_2}{w_3}\right)\left(\frac{d-w_3}{w_4}\right)=m\left(L_f\right).$$
So in the case the link is a rational homology 7-sphere we obtain 
\begin{equation}
a_0 a_1 a_2 a_3 a_4=d-1.
\end{equation}

\item Consider links $L(\textbf{w}, d)$ of the second kind, that it, such that the weight vectors 
$\textbf{w}=(w_{0},w_{1},w_{2},w_{3},w_{4})$ satisfy  $\textbf{w}=(w_{0},w_{1},w_{2},w_{3},w_{4})=(m_{3}v_{0},m_{3}v_{1},m_{2}v_{2},m_{2}v_{3},m_{2}v_{4})$,  $\gcd(m_{2},m_{3})=1$ and $m_{2}m_{3}=d.$ 
%From here, one obtains directly that $u_0=u_1=u_2=m_3$ and 
%$u_3=u_4=m_2.$ Thus, applying the relations $\Lambda_a \Lambda_b=\operatorname{gcd}(a, b) \Lambda_{\operatorname{lcm}(a, b)}$,  
One  obtains  the equality 
$$\operatorname{div}\Delta_f=\alpha({\bf w})\beta({\bf w})\Lambda_d+\beta({\bf w})\Lambda_{m_3}-\alpha({\bf w})\Lambda_{m_2}-\Lambda_1,$$ 
with the two positive  integers $\alpha(\textbf{w})$ and $\beta(\textbf{w})$ depending on the weights:
\begin{equation}
\alpha(\textbf{w})=\frac{m_{2}}{v_{0}v_{1}}-\frac{1}{v_{0}}-\frac{1}{v_{1}}
\end{equation}
and 

\begin{equation}
\beta(\textbf{w})=\left (\frac{m_3}{v_2v_3}-\frac{1}{v_3}-\frac{1}{v_2}\right )\left (\frac{m_3}{v_4}- 1\right ) + \frac{1}{v_4}.
\end{equation}
It is known that if  the link is a rational homology sphere,  then  $\beta({\bf w})=1.$ 

Furthermore, if $f$ has a cycle block of the form $$z_{4}z_{2}^{a_{2}}+z_{2}z_{3}^{a_{3}}+z_{3}z_{4}^{a_{4}},$$ then  
$$
a_2 w_2+w_4=d, \quad a_3 w_3+w_2=d \text { and } a_4 w_4+w_3=d.
$$
From the assumptions on the weight vector $\mathbf{w}$ we have
$$
\begin{aligned}
a_2 a_3 a_4 & =\left(\frac{d-w_4}{w_2}\right)\left(\frac{d-w_2}{w_3}\right)\left(\frac{d-w_3}{w_4}\right) \\
& =\left(\frac{m_3-v_4}{v_2}\right)\left(\frac{m_3-v_2}{v_3}\right)\left(\frac{m_3-v_3}{v_4}\right) \\
& =\frac{m_3^3-\left(v_2+v_3+v_4\right) m_3^2+\left(v_2 v_3+v_2 v_4+v_3 v_4\right) m_3-v_2 v_3 v_4}{v_2 v_3 v_4} \\
& =m_3\left(\frac{m_3^2-\left(v_2+v_3+v_4\right) m_3+v_2 v_3+v_2 v_4+v_3 v_4}{v_2 v_3 v_4}\right)-1 .
\end{aligned}
$$
Since the corresponding link is a $\mathbb{Q}$-homology sphere, it follows that 
$$
\beta(\mathbf{w})=\frac{m_3^2-\left(v_2+v_3+v_4\right) m_3+v_2 v_3+v_2 v_4+v_3 v_4}{v_2 v_3 v_4}=1.
$$
Substituting this equality in the previous equation, we obtain
\begin{equation}
a_2 a_3 a_4+1=m_3.
\end{equation}
\end{enumerate}
\medskip

\noindent{\bf The Berglund-H\"ubsch transpose rule:} 
Recall that the Berglund-Hübsch transpose rule considers an invertible polynomial $$f=\sum_{i=0}^n \prod_{j=0}^n x_j^{a_{i j}}$$ cutting out an orbifold of degree $d$ in $\mathbb{P}(\bf w)$ and defines the transpose polynomial $f^T$ by transposing the exponential  matrix $A=\left(a_{i j}\right)$ of the original polynomial, that is,
$$
f^T=\sum_{i=0}^n \prod_{j=0}^n x_j^{a_{j i}},
$$
again an invertible  polynomial that cuts out an orbifold of degree $\tilde{d}$ in $\mathbb{P}(\tilde{\bf w}).$ Then one considers the links 
$L_f(\mathbf{w}, d)$ and  $L_{f^{T}}(\tilde{\mathbf{w}}, \tilde{d})$ associated to each of these polynomials. 
We will sometimes say that these two links are {\it Berglund-H\"ubsch duals} to one another. 
The following diagram succinctly summarizes the procedure described above, where BH denotes the Berglund-H\"ubsch transpose rule:
\begin{center}
\begin{tikzcd}
{f=0} \arrow{d}\arrow{r}{BH} & {f^T=0}\arrow{d}\\
L_f(\mathbf{w}, d)  \arrow{r}{BH} & L_{f^{T}}(\tilde{\mathbf{w}}, \tilde{d}).\\
\end{tikzcd}
\end{center}
In \cite{CGL} the Berglund-H\"ubsch transpose rule is used to produce Sasaki-Einstein links with the $\mathbb{Q}$-homology of a 7-sphere. In particular, we  found that this rule only produces {\it twins} for singularities of cycle type or of type I and type III. Recall \cite{BGN2, CGL} that  two links $L_f$ and $L_g$ of an isolated hypersurface singularity are called twins if both are $\mathbb{Q}$-homology $(2 n+1)$-spheres and they satisfy 
$m\left(L_f\right)=m\left(L_g\right), d_g=d_f$, and $H_n\left(L_f, \mathbb{Z}\right)=H_n\left(L_g, \mathbb{Z}\right)$.
However, for polynomials of type  II, the Berglund-H\"ubsch transpose rule does not preserve neither torsion nor Milnor number. Their dual links  will receive special attention is Subsection 3.3.

\subsection{Deformations of transverse holomorphic Sasakian structures}
Locally the moduli space of Sasaki isotopy classes is determined by the deformation theory of the transverse holomorphic structure of the foliation $\mathcal{F}_{\xi}.$  So the usual thing to do is to fix the contact structure and deform the transverse holomorphic structure via Kodaira-Spencer theory.

A germ of a deformation of a transverse holomorphic foliation $\mathcal{F}_{\xi}$ on  $M$ with base space $(B, 0)$ is given by an open cover $\left\{U_\alpha\right\}$ of $M$ and a family of local submersions $f_{\alpha, t}: U_\alpha \rightarrow \mathbb{C}^n$ parametrized by $(B, 0)$ that are holomorphic in $t \in B$ for each $x \in U_\alpha$. For $\Theta_{\mathcal{F}_{\xi}}$ denoting the sheaf of transversely holomorphic vector fields on $M$, we have a Kodaira-Spencer map $\rho: T_0 B \rightarrow H^1\left(M, \Theta_{\mathcal{F}_{\xi}}\right)$ that sends $\frac{\partial}{\partial t}$ to a certain class in $H^1\left(M, \Theta_{\mathcal{F}_{\xi}}\right)$ defined by a section $\theta_{\alpha, \beta}$ of the sheaf $\Theta_{\mathcal{F}_{\xi}} \mid U_\alpha \cap U_\beta$. One can consider the full cohomology ring $H^*\left(M, \Theta_{\mathcal{F}_{\xi}}\right)$, these were proven to be finite dimensional. 
In \cite{GHS} it is shown that there is a versal Kuranishi space of deformations given by the map $\Phi: U \rightarrow H^2\left(M, \Theta_{\mathcal{F}_{\xi}}\right)$, for $U$ open set in $H^1\left(M, \Theta_{\mathcal{F}_{\xi}}\right)$, here, as in the complex case, the base of parametrizations is given by $\Phi^{-1}(0)$. We have that if $H^2\left(M, \Theta_{\mathcal{F}_{\xi}}\right)=0$, then the Kuranishi family of deformations of $\mathcal{F}_{\xi}$ is isomorphic to an open set in $H^1\left(M, \Theta_{\mathcal{F}_{\xi}}\right)$. Otherwise, the Kuranishi space may be singular. For a quasiregular Sasakian structure $\mathcal{S}=(\xi, \eta, \Phi, g)$ with quotient orbifold ${Z}$ one obtains the following sequence \cite{BBG} 
$$0 \longrightarrow H^1\left({Z}, \Theta_{{Z}}\right) \longrightarrow H^1\left(M, \Theta_{\mathcal{F}_{\xi}}\right) \longrightarrow H^0\left({Z}, \Theta_{{Z}}\right) \longrightarrow H^2\left({Z}, \Theta_{{Z}}\right)$$ where  $\Theta_{{Z}}$ denotes the sheaf of germs of holomorphic vector fields on the complex orbifold ${Z}$ and $H^0\left({Z}, \Theta_{{Z}}\right)$ can be considered as the Lie algebra of the group of holomorphic automorphisms of ${Z}.$  Thus the deformation of the transverse holomorphic structures of the foliation can be understood in terms of the deformation of the complex structure of the orbifold, which are described by $H^1\left({Z}, \Theta_{{Z}}\right)$ and in terms of the deformations of the Reeb vector in  $H^0\left({Z}, \Theta_{{Z}}\right)$ which are described by the Sasaki cone \cite{BGS}. Of course, if $H^0\left({Z}, \Theta_{{Z}}\right)=0$ we obtain an isomorphisms between  $H^1\left (M, \Theta_{\mathcal{F}_{\xi}}\right )$  and $H^1\left({Z}, \Theta_{{Z}}\right).$ 
Deformations of  the transversely holomorphic structure may not remain Sasakian unless the 
$(0,2)$ component of the basic Euler class $\left[d \eta^{0,2}\right] \in H^{(0,2)}_B(M, \mathcal{F}_{\xi})$ is zero, where $H^{0,2}_B(M, \mathcal{F}_{\xi})$ is the basic Dolbeault cohomology for the transversal complex structure. However, in case the Sasakian structure is positive, which is the case we are interested in this article, it has been proven in \cite{No} that $H^{0,q}_B(M, \mathcal{F}_{\xi})=0$ for $q>0.$ 

We pay particular attention to the local  moduli of quasiregular Sasakian structures on links of isolated singularities, so we focus on orbifolds that are  quasismooth weighted hypersurfaces in certain weighted projective spaces. Moreover, all the orbifolds under discussion in this article have finite automorphism group, so we will assume $H^0\left({Z}, \Theta_{{Z}}\right)=0.$ The proof of the  next  two theorems can be found in \cite{BBG}. (See also \cite{BGK}.)

\begin{theo} Let $Z_f$ be a quasismooth weighted hypersurface in $\mathbb{P}(\mathbf{w})$ corresponding to the weighted homogenous polynomial $f$ of degree $d$ and weight vector $\mathbf{w}=\left(w_0, \ldots, w_n\right)$ with $H^0\left(Z_f, \Theta_{Z_f}\right)=0.$ Assume also that $n \geq 3$. Then the complex orbifolds $Z_f$ form a continuous family of complex dimension
\begin{equation}
\dim_{\mathbb C} H^{0}(\mathbb{P}(\textbf{w}),\mathcal{O}({d}))-\sum_{i}\dim_{\mathbb C} H^{0}(\mathbb{P}(\mathbf{w}),\mathcal{O}(w_{i})).
\end{equation}
Furthermore, if  the index $I=|\mathbf{w}|-d>0$  and $Z_f$ admits a Kähler-Einstein metric for a generic $f$ then it admits a $2\left[h^0(\mathbb{P}(\mathbf{w}), \mathcal{O}(d))-\sum_i h^0\left(\mathbb{P}(\mathbf{w}), \mathcal{O}\left(w_i\right)\right)\right]$ dimensional family of Kähler-Einstein metrics up to homothety.
 \end{theo}

For quasiregular Sasaki-Einstein metrics we have the following: 

\begin{theo}  Let $M$ be a smooth compact manifold and let $\mathcal{S}=\left(\xi, \eta, \Phi_t, g_t\right)$ be a family of quasiregular Sasaki-Einstein structures on $M$ induced by a continuous family of inequivalent complex orbifolds $\mathcal{Z}_t$ with Kähler-Einstein metrics. Then the metrics $g_t$ are inequivalent as Sasaki-Einstein metrics. 
\end{theo}

From these two theorems it is clear that one can compute the dimension of the local moduli of Sasaki-Einstein metrics through the information given by the moduli of K\"ahler-Einstein metrics on Fano orbifolds. We will do this for rational homology 7-spheres that are obtained as links of invertible polynomials.

 The procedure to determine the moduli of  K\"ahler-Einstein metrics on Fano orbifolds  goes as follows
 \begin{enumerate}
     \item[(a)]  Consider $X_d \subset \mathbb{P}\left(w_0, \ldots, w_{n}\right)$ a well-formed and quasismooth weighted projective hypersurface of degree $d$ with finite automorphism group, 
     cut out by an invertible polynomial $f$ of certain defined type. 
     \item[(b)] Determine the linear system 
 $${\mathcal{X}}_d=\left|\mathcal{O}_{\mathbb{P}(\bf w)}(d)\right|=\mathbb{P}H^0(\mathbb{P}(\mathbf{w}), \mathcal{O}(d)),$$ 
 that is, the parameter space for all hypersurfaces of degree $d$ in $\mathbb{P}(\mathbf w).$ 
 \item[(c)] Determine the automorphism group $\mathcal{G}(\bf w)$ of $ \mathbb{P}(\mathbf w).$ 
 \item[(d)] Since $\mathcal{G}(\bf w)$  acts on $\left|\mathcal{O}_{\mathbb{P}(\bf w)}(d)\right|$ and there is an inclusion of the  set $\mathcal{X}_d^{QS}$ of quasismooth hypersurfaces of degree $d$ in $\mathbb{P}(\bf w),$  as an open set in $\mathcal{X}_d,$ one obtains  the quotient 
 $$\mathcal{X}_d^{QS}/\mathcal{G}(\bf w),$$ which is a coarse moduli space, see \cite{KeM} Corollary 1.2. 
 Actually, since $$\mathcal{X}_d^{QS}/\mathcal{G}({\bf{w}})\subset\mathbb{P}H^0(\mathbb{P}({\bf{w}}), \mathcal{O}(d))/\mathbb{P}G({\bf{w}})=H^0(\mathbb{P}({\bf{w}}), \mathcal{O}(d))/G({\bf{w}}),$$ where $G(\mathbf{w})$ is the group of automorphisms of the graded ring $S(\mathbf{w}),$ we will give a precise description of the former quotient. 
 \item[(e)] Additionally, we will assume that the weighted hypersurface $X_d\subset \mathbb{P}\bf(w)$ satisfies the estimate $d I<\frac{n}{(n-1)} \min _{i, j}\left\{w_i w_j\right\}$ in Theorem 2.1.1, which implies the existence of  K\"ahler-Einstein metrics in all the elements in $\mathcal{X}_d^{QS}.$ It follows from Theorems 2.2.1 and 2.2.2 that the number of parameters of the space of inequivalent Sasaki-Einstein metrics on the corresponding link $L_f=V_f\cap S^{2n+1}$  is given by the dimension of the moduli  $\mathcal{X}_d^{QS}/\mathcal{G}(\bf w).$ 
\end{enumerate}

Notice that well-formedness is required, otherwise we only obtain upper bounds in the dimension of the moduli space. For instance consider $Z_9\subset \mathbb{P}(3,3,3,3,3)$ a hypersurface of degree 9 which is isomorphic as a variety to   $Z_3\subset \mathbb{P}(1,1,1,1,1)$ a hypersurface of degree 3 with the same equation. In Subsection 3.3, where we study the moduli of the Berglund-H\"ubsch duals, we  deal with non well-formed weighted hypersurface where the procedure descried above  suffices to determine the precise dimension of the moduli in most cases. In Section 4 we also discuss the role of the boundary points of $\overline{\mathcal{X}_d^{QS}/\mathcal{G}(\bf w)}$, that is, the non-quasismooth weighted varieties,  in the construction of Sasaki-Einstein metrics on non-smooth links.

% !TEX root = MainMod.tex
\section{Local moduli of Sasaki-Einstein metrics  on smooth rational homology 7-spheres}
In this section we will study the space of deformations of polynomials of cycle type, type I, type II and type III. We focus only on these types of singularities for the reasons that were argued in Section 2.1: if a link $L_f$ is a smooth rational homology 7-sphere and admits Sasaki-Einstein metric then $f$ has to be one of the aforementioned types. In the process  we find the monomials generating $H^0(\mathbb{P}(\mathbf{w}), \mathcal{O}(d))$ and the group of automorphisms of the weighted projective space 
$\mathbb{P}(\mathbf{w})$ that contains the corresponding orbifold. 

\subsection{Rational homology 7-spheres: the cycle type}

Let $f$ be an invertible polynomial of the form
\begin{equation}    f=z_{4}z_{0}^{a_{0}}+z_{0}z_{1}^{a_{1}}+z_{1}z_{2}^{a_{2}}+z_{2}z_{3}^{a_{3}}+z_{3}z_{4}^{a_{4}}
\end{equation}
of degree $d$ and associated weight vector $\textbf{w} =(w_{0},w_{1},w_{2},w_{3},w_{4})$ with  $\gcd(d,w_{i})=1,$ which in particular implies that the weighted hypersurface $Z_f\subset \mathbb{P}(\bf{w})$ is well-formed.
From (3.1.1), we have the following relation between the weights $w_{i}$'s and the degree:
\begin{equation}
    a_{0}w_{0}+w_{4}=d, \ \ a_{1}w_{1}+w_{0}=d, \ \ a_{2}w_{2}+w_{1}=d, \ \ a_{3}w_{3}+w_{2}=d, \ \ a_{4}w_{4}+w_{3}=d.
\end{equation}
Moreover, if the link $L_{f}$ is a rational homology sphere, from the equality (2.1.1) $d=1+a_{0}a_{1}a_{2}a_{3}a_{4}$, we can express each weight $w_{i}$ as:
\begin{equation}
    w_{i}=1-a_{i+1}+a_{i+1}a_{i+2}-a_{i+1}a_{i+2}a_{i+3}+a_{i+1}a_{i+2}a_{i+3}a_{i+4},
\end{equation}
where the subscript is taken mod $5$. Furthermore, considering the relations given in (3.1.2) and the fact that $\gcd(d,w_{i})=1$, we conclude that two consecutive weights  $w_{i}$ and $w_{i+1}$ are always co-prime. 

\begin{rem} Notice that the assumption (e) below Theorem 2.2.2 and the relations in (3.1.2) imply that   $a_i>1$ for $i=0, \ldots ,  4.$  This fact  will be used throughout the proof of Lemma 3.1.2. Indeed, without loosing generality,  let us assume that $a_2=1.$ Then the third equation in (3.1.2)  gives  $w_1+w_2=d$ and  the estimate $dI< \frac{4}{3}\min _{i, j}\left\{w_i w_j\right\}$ can be written as  $$(w_1+w_2)(w_0+w_3+w_4)<\frac{4}{3}\min _{i, j}\left\{w_i w_j\right\}.$$ Since 
$6\min _{i, j}\left\{w_i w_j\right\}\leq (w_0+w_3+w_4)(w_1+w_2)$ we obtain a contradiction.
\end{rem}

Now, we find the generators of the space of deformations of the orbifold $Z_f$. In  \cite{BGK}, it was proven  that the automorphism group of any orbifold  $Z_f$ is finite as long as $w_i< \frac{1}{2}d$ for all but one of the $w_i's$ for $f$ quasismooth.  Since $\gcd(d,w_{i})=1$, the weight vector $\bf w$ does not admit polynomials that contain blocks of the form $z_{i}^{2}+z_{j}^{2}$, thus   
$H^0(Z_f, \Theta_{Z_f})=0.$ So the complex dimension of the moduli of $Z_f$ is given by 
\begin{equation}
    \dim_{\mathbb C} H^{0}(\mathbb{P}(\textbf{w}),\mathcal{O}({d}))-\sum_{i}\dim_{\mathbb C} H^{0}(\mathbb{P}(\mathbf{w}),\mathcal{O}(w_{i})). 
\end{equation}

Let us compute $\dim_{\mathbb C} H^{0}(\mathbb{P}(\textbf{w}),\mathcal{O}({d}))$. Here we consider all possible monomials $z_{0}^{x_{0}}z_{1}^{x_{1}}z_{2}^{x_{2}}z_{3}^{x_{3}}z_{4}^{x_{4}}$ of degree $d$. Notice that it is equivalent to solving the following Diophantine equation
\begin{equation}  w_{0}x_{0}+w_{1}x_{1}+w_{2}x_{2}+w_{3}x_{3}+w_{4}x_{4}=d
\end{equation}
with variables $x_{i}\in\mathbb{Z}_{0}^{+}$ and where at least one of them is nonzero. 

First recall a well-known result for Diophantine equations that will be used in the arguments that follow.
\begin{lem}
 We consider the Diophantine equation in variables $x,y$:
    $$ax+by=c$$
If $\gcd(a,b)=1$ and $(x_{0},y_{0})$ is a particular solution, then all the solutions of the Diophantine equation are given by
$$(x,y)=(x_{0}-bk,y_{0}+ak), \ \ \mbox{ where } k\in\mathbb{Z}.$$   
\end{lem}

The next lemma determines all solutions of Equation (3.1.5).
\begin{lem}
    Let $\bf w$ be defined as in (3.1.2). Then the Diophantine equation (3.1.5) has exactly five solutions. These solutions are $$(a_{0},0,0,0,1), (1,a_{1},0,0,0), (0,1,a_{2},0,0), (0,0,1,a_{3},0) \mbox{ and } (0,0,0,1,a_{4}).$$
    Thus, the set of generators of the space   
$H^{0}(\mathbb{P}(\textbf{w}),\mathcal{O}(d))$  is  given by $$\{z_0^{a_0}z_4, z_0z_1^{a_1}, z_1z_2^{a_2}, z_2z_3^{a_3}, z_3z_4^{a_4}\}.$$

    \begin{proof}
        Since the polynomial $f$ associated to $\bf w$ is cycle, we can assume, without loss of generality, that $w_{4}=\min_{i}w_{i}$.
        
        Now, in Equation (3.1.5), as $\gcd(w_{3},w_{4})=1$, we can define a new variable $y_{3}=w_{3}x_{3}+w_{4}x_{4}\geq0$. Replacing this in Equation 
        (3.1.5), we obtain a new Diophantine equation
        \begin{equation*}
        w_{0}x_{0}+w_{1}x_{1}+w_{2}x_{2}+y_{3}=d.   
        \end{equation*}
        Here, we define the variable $y_{2}=w_{2}x_{2}+y_{3}\geq0$. Putting this in the equation above, we get
        \begin{equation*}
         w_{0}x_{0}+w_{1}x_{1}+y_{2}=d.    
        \end{equation*}
        Repeating the previous process, we define the new variable $y_{1}=w_{1}x_{1}+y_{2}\geq0$. Again, replacing this above, we obtain
        \begin{equation}
         w_{0}x_{0}+y_{1}=d.  
        \end{equation}
        From (3.1.2), we obtain a particular solution $(a_{0},w_{4})$ of Equation (3.1.6). Thus, the general solution is given by
        $$x_{0}=a_{0}-t \ \ \mbox{ and } \ \ y_{1}=w_{4}+tw_{0}, \ \ \mbox{ where }t\in\mathbb{Z}.$$
        
        Now, we solve the Diophantine equation 
        \begin{equation}
          w_{1}x_{1}+y_{2}=y_{1}.  
        \end{equation}
        Since  $y_{1}=w_{4}+tw_{0}$ and $w_{0}=d-a_{1}w_{1}$, we obtain  
        $$w_{1}x_{1}+y_{2}=y_{1}=w_{4}+t(d-a_{1}w_{1})=-ta_{1}w_{1}+w_{4}+td.$$
        Then we have a particular solution $(-ta_{1},w_{4}+td)$. Therefore, the general solution of (3.1.7) is given by
        $$x_{1}=-ta_{1}-s \ \ \mbox{ and } \ \ y_{2}=w_{4}+td+sw_{1}, \ \ \mbox{where }t,s\in\mathbb{Z}$$
        
        On the other hand, to solve the Diophantine equation
        \begin{equation}
            w_{2}x_{2}+y_{3}=y_{2}
        \end{equation}
        we take into account that $y_{2}=w_{4}+td+sw_{1}$ and $w_{1}=d-a_{2}w_{2}$. Thus, we obtain
        $$w_{2}x_{2}+y_{3}=y_{2}=w_{4}+td+sw_{1}=-sa_{2}w_{2}+w_{4}+(t+s)d$$
        Then, a particular solution of equation above is $(-sa_{2},w_{4}+(t+s)d).$ So the general equation is
        $$x_{2}=-sa_{2}-r \ \ \mbox{ and } \ \ y_{3}=w_{4}+(t+s)d + rw_{2} \ \ \mbox{ where }r,s,t\in\mathbb{Z}$$
        
        Finally, we will solve the Diophantine equation
        \begin{equation}
            w_{3}x_{3}+w_{4}x_{4}=y_{3}
        \end{equation}
        Since $y_{3}=w_{4}+(t+s)d+rw_{2}$ and $w_{2}=d-a_{3}w_{3}$, we have
        $$w_{3}x_{3}+w_{4}x_{4}=y_{3}=w_{4}+(t+s)d+rw_{2}=(t+s+r-ra_{3})w_{3} + (1+a_{4}(t+s+r))w_{4}$$
        This implies that a particular solution of (3.1.9) is $(t+s+r-ra_{3},1+a_{4}(t+s+r))$. Then, the general solution is given by
        $$x_{3}=t+s+r-ra_{3}-qw_{4} \ \ \mbox{ and } \ \ x_{4}=1+(t+s+r)a_{4}+qw_{3} \ \mbox{ where }q,r,s,t\in\mathbb{Z}$$
        
        Considering all of the above, we have the general solution for (3.1.5):
        {\small{
        \begin{equation}
           (x_{0},x_{1},x_{2},x_{3},x_{4})=(a_{0}-t, -ta_{1}-s, -sa_{2}-r, t+s+r-ra_{3}-qw_{4},1+(t+s+r)a_{4}+qw_{3}) 
        \end{equation}}}
        where $q,r,s,t\in\mathbb{Z}$. Notice that in order to solve (3.1.5), we require $x_{i}\geq0$ for all $i$. This implies that
        \begin{equation}
            t\leq a_{0}, \ \ \  s\leq -ta_{1}, \ \ \ r\leq -sa_{2}, \ \ \ ra_{3}+qw_{4} \leq t+s+r \ \ \mbox{ and } \ \ -qw_{3}-1\leq (t+s+r)a_{4}.
        \end{equation}
        
        On the other hand, for the variables $y_{i}$'s, we have 
        \begin{equation}
            (y_{1},y_{2},y_{3})=(w_{4}+tw_{0}, w_{4}+td+sw_{1}, w_{4}+td+sd+rw_{2}), \ \ \mbox{ where }r,s,t\in\mathbb{Z}
        \end{equation}
         As $x_{i}\geq0$, then we have $y_{i}\geq0$. So we obtain  the following inequalities
         \begin{equation}
             tw_{0}+w_{4}\geq 0, \ \ \ w_{4}+td+sw_{1}\geq0 \ \ \mbox{ and } \ \ w_{4}+td+sd+rw_{2}\geq 0.
         \end{equation}

        \underline{\textbf{Claim 1}}: $q\in\{-1,0\}$. First, we prove that $q\leq0$. Indeed, from (3.1.11) we have 
        \begin{align*}
            x_{3}=t+s+r-ra_{3}- qw_{4} \geq 0 & \Rightarrow t(1-a_{3})+s(1-a_{3})+r(1-a_{3}) \geq qw_{4}-ta_{3}-sa_{3} \\
            & \Rightarrow t+s+r \leq \dfrac{ta_{3}+sa_{3}-qw_{4}}{a_{3}-1}. 
        \end{align*}
        Also, from (3.1.11), we obtain
        $$ t+s+r \geq \dfrac{-qw_{3}-1}{a_{4}}.$$
        From the two inequalities above, we get
        $$\dfrac{-qw_{3}-1}{a_{4}}\leq \dfrac{ta_{3}+sa_{3}-qw_{4}}{a_{3}-1} \Rightarrow -qa_{3}w_{3}-a_{3}+qw_{3} +1 \leq ta_{3}a_{4}+sa_{3}a_{4}-qa_{4}w_{4}.$$
        Using (3.1.2), we have $a_{3}w_{3}=d-w_{2}$ and $a_{4}w_{4}=d-w_{3}$. Replacing these in the previous inequality  and simplifying, we obtain
        \begin{equation}
           qw_{2}\leq a_{3}-1+ta_{3}a_{4}+sa_{3}a_{4}.
        \end{equation}
     Considering the equality (3.1.3) for $w_{2}$, we have
    \begin{align*}
        qw_{2} & \leq a_{3} -1+a_{0}a_{3}a_{4}-ta_{1}a_{3}a_{4} \\ 
        & = -1 +a_{3} -a_{3}a_{4} +a_{0}a_{3}a_{4}-a_{0}a_{1}a_{3}a_{4} +(a_{0}-t)a_{1}a_{3}a_{4} +a_{3}a_{4} \\
        & = -w_{2} +(a_{0}-t)a_{1}a_{3}a_{4} +a_{3}a_{4}.
    \end{align*}
    On the other hand, since $\gcd(w_{0},w_{4})=1$ and $w_{4}$ is the minimum of all $w_{i}$'s, we have $w_{4}<w_{0}$. Moreover, using (3.1.13), we obtain $t\geq -\dfrac{w_{4}}{w_{0}}> -1$. Thus, we have $t\geq0$.  Thus 
    $$qw_{2}\leq -w_{2}+(a_{0}-t)a_{1}a_{3}a_{4} +a_{3}a_{4}\leq -w_{2}+a_{0}a_{1}a_{3}a_{4}+a_{3}a_{4}$$ which implies 
    \begin{equation}
        q\leq -1 +\dfrac{a_{0}a_{1}a_{3}a_{4}+a_{3}a_{4}}{w_{2}}.
    \end{equation}
    Now, we will show that $2w_{2}>a_{0}a_{1}a_{3}a_{4}+a_{3}a_{4}$. Indeed, we have 
    \begin{align*}
     2w_{2}>a_{0}a_{1}a_{3}a_{4}+a_{3}a_{4}  & \Longleftrightarrow 2(1-a_{3}+a_{3}a_{4}-a_{0}a_{3}a_{4}+a_{0}a_{1}a_{3}a_{4}) > a_{0}a_{1}a_{3}a_{4}+a_{3}a_{4} \\
     & \Longleftrightarrow a_{0}a_{3}a_{4}(a_{1}-2)+a_{3}(a_{4}-2)+2>0.
    \end{align*}
    From Inequality  (3.1.15) we have   
    $$q\leq -1 +\dfrac{a_{0}a_{1}a_{3}a_{4}+a_{3}a_{4}}{w_{2}} <-1+2 = 1.$$
    Therefore, we have $q\leq 0$.

    Next, we will show that $q\geq -1$. First, we will prove that $d<2a_{1}w_{1}$. We remember that $d=1+a_{0}a_{1}a_{2}a_{3}a_{4}$ and using (3.1.3) we write $w_{1}=1-a_{2}+a_{2}a_{3}-a_{2}a_{3}a_{4}+a_{0}a_{2}a_{3}a_{4}$. Then
    \begin{align*}
        d < 2a_{1}w_{1} & \Longleftrightarrow 1+a_{0}a_{1}a_{2}a_{3}a_{4} < 2a_{1}(1-a_{2}+a_{2}a_{3}-a_{2}a_{3}a_{4}+a_{0}a_{2}a_{3}a_{4}) \\
        & \Longleftrightarrow 1 < 2a_{1} -2a_{1}a_{2}+2a_{1}a_{2}a_{3}-2a_{1}a_{2}a_{3}a_{4}+a_{0}a_{1}a_{2}a_{3}a_{4} \\
        & \Longleftrightarrow 1< 2a_{1}+2a_{1}a_{2}(a_{3}-1) +a_{1}a_{2}a_{3}a_{4}(a_{0}-2).
    \end{align*}
Since $a_{3}\geq2$ and $a_{0}\geq2$, we conclude that $d<2a_{1}w_{1}$. 
Now, from (3.1.11) and (3.1.13), we have
$$sa_{2}w_{2}\leq -rw_{2} \leq w_{4}+td+sd.$$
Replacing $a_{2}w_{2}=d-w_{1}$ in the above inequality, we obtain $-sw_{1}\leq w_{4}+td $. Moreover, as $d<2a_{1}w_{1}$ and $w_{4}=\min_{i}w_{i}  \leq w_{1}$, we get
$$-sw_{1}\leq w_{1}+td < w_{1}+2ta_{1}w_{1} \Rightarrow -s < 1+2ta_{1} \Rightarrow -s\leq 2ta_{1}.$$
 By (3.1.11), $a_{0}\geq t$, then we have
    \begin{equation}
        -s \leq 2ta_{1} \leq 2a_{0}a_{1}.
    \end{equation}
Also, from (3.1.11) we have $1+(t+s+r)a_{4}+qw_{3}\geq0$, $-sa_{2}\geq r$ and $t\leq a_{0}$. Then
$$-qw_{3}\leq 1+a_{4}(t+s+r) \leq 1+a_{4}(t-s(a_{2}-1))\leq 1+a_{4}(a_{0}-s(a_{2}-1)).$$
Replacing (3.1.16) in the above inequality, we obtain
$$-qw_{3} \leq 1+a_{4}(a_{0}+2a_{0}a_{1}(a_{2}-1)) = 1+a_{0}a_{4}-2a_{0}a_{1}a_{4}+2a_{0}a_{1}a_{2}a_{4}.$$
Adding $1+a_{0}a_{4}-2a_{4} =1+a_{4}(a_{0}-2)>0$ to the right of the last inequality, we have 
$$-qw_{3}<2(1-a_{4}+a_{4}a_{0}-a_{4}a_{0}a_{1}+a_{4}a_{0}a_{1}a_{2})=2w_{3} \Rightarrow q>-2 \Rightarrow q\geq -1. $$
So  $q\in \{ -1,0\}$.

Next, we will determine the number of solutions for the two values of $q$. 
\begin{itemize}
    \item[a)] If $q=0$, we have the unique solution $(a_{0},0,0,0,1).$ 
    Indeed from (3.1.11) we have 
    \begin{equation}
        t\leq a_{0}, \ \ \  s\leq -ta_{1}, \ \ \ r\leq -sa_{2}, \ \ \ ra_{3} \leq t+s+r \ \ \mbox{ and } \ \ -1\leq (t+s+r)a_{4}.
    \end{equation}
    Since $t+s+r-ra_{3}\geq 0$, we obtain $(t+s+r)(1-a_{3})\geq -sa_{3}-ta_{3}$ which implies 
    \begin{equation}
        t+s+r\leq \dfrac{a_{3}(s+t)}{a_{3}-1}.
    \end{equation}
    On the other hand, from the last inequality in (3.1.17), we have $t+s+r\geq -\dfrac{1}{a_{4}}>-1$ which means  $t+s+r\geq 0$. So from Inequality  (3.1.18) we have $s+t\geq0$.
  
 Now from $\gcd(w_{4},w_{0})=1$ and $w_{4}=\min_{i}w_{i}$, we have $w_{4}<w_{0}$. Then $-1<-\dfrac{w_{4}}{w_{0}}\leq t \leq a_{0}$. So we obtain $0\leq t \leq a_{0}$. In addition, from (3.1.17) we have $s\leq -ta_{1}\leq0$. As $-t\leq s$, we have 
    $$s\leq -ta_{1} \leq sa_{1} \Rightarrow 0\leq s(a_{1}-1) \Rightarrow s\geq0.$$
    Since $s\leq0,$ we obtain $s=0$ so  $t=0$. Finally, from (3.1.17) we have $r\leq -sa_{2}=0$ and since $t=s=0$, we have $r=t+s+r\geq0$. Thus,  $r=0$. Then, we have a solution for (3.1.5): $(a_{0},0,0,0,1)$.
    
    \item[b)] If $q=-1$, we have to analyze for cases: In  (3.1.11) we have 
    \begin{equation}
            t\leq a_{0}, \ \ \  s\leq -ta_{1}, \ \ \ r\leq -sa_{2}, \ \ \ ra_{3}-w_{4} \leq t+s+r \ \ \mbox{ and } \ \ w_{3}-1\leq (t+s+r)a_{4}.
        \end{equation}
     \underline{\textbf{Claim 2}}: $t\in\{a_{0}-1,a_{0}\}$. Indeed, from (3.1.3) we write $w_{3}=1-a_{4}+a_{4}a_{0}-a_{4}a_{0}a_{1}+a_{4}a_{0}a_{1}a_{2}$. Replacing this in the last inequality in (3.1.19) we obtain
    \begin{equation}
        t+s+r\geq \dfrac{w_{3}-1}{a_{4}} = \dfrac{-a_{4}+a_{4}a_{0}-a_{4}a_{0}a_{1}+a_{4}a_{0}a_{1}a_{2}}{a_{4}} = -1+a_{0}-a_{0}a_{1}+a_{0}a_{1}a_{2}.
    \end{equation}
    From (3.1.19), using $a_{0}-t\geq0$ and $-sa_{2}\geq r$  in (3.1.20), we obtain 
    \begin{equation}
      -sa_{2}+s\geq r+s \geq -1+a_{0}-t-a_{0}a_{1}+a_{0}a_{1}a_{2} \geq -1-a_{0}a_{1}+a_{0}a_{1}a_{2}.
    \end{equation}
    Since  $1-a_{2}<0$, then
    $$s\leq\dfrac{ -1-a_{0}a_{1}(1-a_{2})}{1-a_{2}} = \dfrac{1}{a_{2}-1}-a_{0}a_{1}\leq 1-a_{0}a_{1}.$$
   Actually $s<1-a_{0}a_{1}.$ Indeed 
   if   $s=1-a_{0}a_{1}$ (which would force  $a_{2}=2$) from  (3.1.21) we can write 
    $$-s\geq r+s\geq -1+a_{0}-t+a_{0}a_{1}\geq -1+a_{0}a_{1}= -s.$$
    In this inequality, we have that $t=a_{0}$. Thus, in the second inequality of (3.1.19) we obtain $s\leq -a_{0}a_{1}$, which contradicts the assumption $s=1-a_{0}a_{1}$. So we have $s<1-a_{0}a_{1}$ so we say 
    \begin{equation}
        s\leq -a_{0}a_{1}.
    \end{equation}
    Replacing (3.1.22) in (3.1.20), we obtain
    \begin{equation}
        r \geq -1+(a_{0}-t)+(-s-a_{0}a_{1})+a_{0}a_{1}a_{2}\geq -1+a_{0}a_{1}a_{2}.
    \end{equation}
    On the other hand, from (3.1.13)  and (3.1.22) we obtain 
    $$w_{4}+td\geq -sw_{1} \geq a_{0}a_{1}w_{1}.$$
    Since $a_{1}w_{1}=d-w_{0}$ and $w_{4}=d-a_{0}w_{0}$ the previous inequality can be rewritten as 
    $$a_{0}(d-w_{0}) \leq d-a_{0}w_{0}+td \Rightarrow a_{0}-1\leq t.$$
    Also, from (3.1.19) we have  $t\leq a_{0}$. Thus  $t\in\{ a_{0}-1,a_{0}\}$. 
    \medskip
    
    Next, we will determine the solutions of the Diophantine equation (3.1.5) for each value of $t$.
    \begin{itemize}
        \item[i)]  If $t=a_{0}-1$, from (3.1.13) we obtain  
        $$w_{4}+(a_{0}-1)d+sd+rw_{2}\geq0 \Rightarrow w_{4}-d+a_{0}d+sd\geq -rw_{2}.$$
        From (3.1.19) we have $r\leq -sa_{2}$. In addition, the weights verify $d=w_{4}+a_{0}w_{0}=w_{1}+a_{2}w_{2}$. Thus,  the inequality above can be written 
        $$-a_{0}w_{0}+a_{0}d+sd\geq -rw_{2}\geq sa_{2}w_{2} = sd-sw_{1}.$$ 
        As $a_{1}w_{1}=d-w_{0}$, then $-a_{0}a_{1}\leq s$. Also, by (3.1.22) we know that $s\leq -a_{0}a_{1}$. Hence  $s=-a_{0}a_{1}$. On the other hand, replacing $t=a_{0}-1$ and $s=-a_{0}a_{1}$ in (3.1.20), we obtain $r\geq a_{0}a_{1}a_{2}$.  
        From (3.1.19) we have
        $$ra_{3}-w_{4}\leq t+s+r \Rightarrow r \leq \dfrac{w_{4}+t+s}{a_{3}-1}= \dfrac{w_{4}+(a_{0}-1)+(-a_{0}a_{1})}{a_{3}-1}.$$
        As $w_{4}=1-a_{0}+a_{0}a_{1}-a_{0}a_{1}a_{2}+a_{0}a_{1}a_{2}a_{3}$, we obtain $r\leq a_{0}a_{1}a_{2}$. Thus, $r=a_{0}a_{1}a_{2}$. In this case, from (3.1.10) it follows that a solution for (3.1.5) is $(1,a_{1},0,0,0)$.
        \item[ii)] If $t=a_{0}$. Following a similar process that in i) and using (3.1.19) and (3.1.22) we have
        \begin{equation}
            r\leq \dfrac{t+s+w_{4}}{a_{3}-1} = \dfrac{a_{0}+s+w_{4}}{a_{3}-1}\leq \dfrac{a_{0}-a_{0}a_{1}+w_{4}}{a_{3}-1}.
        \end{equation}
         Moreover, since $w_{4}=1-a_{0}+a_{0}a_{1}-a_{0}a_{1}a_{2}+a_{0}a_{1}a_{2}a_{3}$, we obtain
        $$r\leq \dfrac{a_{0}-a_{0}a_{1}+w_{4}}{a_{3}-1}\leq \dfrac{1-a_{0}a_{1}a_{2}+a_{0}a_{1}a_{2}a_{3}}{a_{3}-1}= \dfrac{1}{a_{3}-1}+a_{0}a_{1}a_{2}.$$
        When $a_{3}=2$, it is possible that $r=1+a_{0}a_{1}a_{2}$. Let us see that this situation cannot happen. Indeed, if we assume that $r=1+a_{0}a_{1}a_{2}$, then $a_{3}=2$. It implies that $w_{4}=1-a_{0}+a_{0}a_{1}+a_{0}a_{1}a_{2}$. Replacing in (3.1.24), we obtain
        $$r\leq 1+a_{0}a_{1}+s+a_{0}a_{1}a_{2}\leq 1+a_{0}a_{1}a_{2} =r.$$
        Thus, we get $s=-a_{0}a_{1}$. Putting this in the third inequality of (3.1.19), we have $r\leq a_{0}a_{1}a_{2}$, which results in a contradiction. Therefore, we have $r\leq a_{0}a_{1}a_{2}$. Also, from (3.1.23), it verifies  $r\geq -1+a_{0}a_{1}a_{2}$. Hence, we have $r\in\{ -1+a_{0}a_{1}a_{2},a_{0}a_{1}a_{2}\}$. Next, we detail each case.
        \begin{itemize}
            \item If $r=-1+a_{0}a_{1}a_{2}$, we have in the last inequality of (3.1.19):
            $$a_{0}+s-1+a_{0}a_{1}a_{2} = t+s+r\geq \dfrac{w_{3}-1}{a_{4}}=-1+a_{0}-a_{0}a_{1}+a_{0}a_{1}a_{2}, $$
            which implies that $s\geq -a_{0}a_{1}$. Using (3.1.22), we have $s=-a_{0}a_{1}$. Thus, in this case  from (3.1.10) it follows that the solution of (3.1.5) is $(0,0,1,a_{3},0)$. 
            
            \item If $r=a_{0}a_{1}a_{2}$, in a similar way as we have worked above, in the last inequality of (3.1.19) we have 
             $$a_{0}+s+a_{0}a_{1}a_{2} = t+s+r\geq \dfrac{w_{3}-1}{a_{4}}=-1+a_{0}-a_{0}a_{1}+a_{0}a_{1}a_{2}, $$
            which implies that $s\geq -1-a_{0}a_{1}$. From (3.1.22), $s\leq -a_{0}a_{1}$, we obtain $s=-1-a_{0}a_{1}$ or $s=-a_{0}a_{1}$. If $s=-1-a_{0}a_{1}$, the solution for (3.1.5) is $(0,1,a_{2},0,0)$. On the other hand, if $s=-a_{0}a_{1}$,  from (3.1.10) it follows  that  $(0,0,0,1,a_{4})$ is the solution for the equation (3.1.5).
        \end{itemize}
    \end{itemize}
\end{itemize}
Therefore, the equation (3.1.5) has exactly five solutions:
$$(a_{0},0,0,0,1), (1,a_{1},0,0,0), (0,1,a_{2},0,0), (0,0,1,a_{3},0) \mbox{ and } (0,0,0,1,a_{4}).$$
    \end{proof}
\end{lem}

Next, we compute the generators of  $H^{0}(\mathbb{P}(\mathbf{w}),\mathcal{O}(w_{i}))$. For this, we need to find all monomials $z_{0}^{a_{0}}z_{1}^{a_{1}}z_{2}^{a_{2}}z_{3}^{a_{3}}z_{4}^{a_{4}}$ of degree $w_i$. It is equivalent to solving the Diophantine equation
\begin{equation}  w_{0}x_{0}+w_{1}x_{1}+w_{2}x_{2}+w_{3}x_{3}+w_{4}x_{4}=w_{i}
\end{equation}
with variables  $x_{i}\in\mathbb{Z}_{0}^{+}$ and where at least one of them is nonzero.
\begin{lem}
    The equation (3.1.25) has a unique solution for each $w_{i}$ and the set of generators of  $H^{0}(\mathbb{P}(\mathbf{w}),\mathcal{O}(w_{i}))$ is given by  $\{z_i\}$ for $i=0\ldots , 4.$
    \begin{proof}
      Since $f$ is a cycle polynomial, we can work without loss of generality with $w_{i}=w_{1}$. Thus, the equation (3.1.25) results  in  
      \begin{equation}  w_{0}x_{0}+w_{1}x_{1}+w_{2}x_{2}+w_{3}x_{3}+w_{4}x_{4}=w_{1}.
\end{equation}
From (3.1.2), we have $a_{2}w_{2}+w_{1}=d$. Then, if we add $a_{2}w_{2}$ to both sides of (3.1.26), we arrive to a new Diophantine equation
\begin{equation}  w_{0}x_{0}+w_{1}x_{1}+w_{2}\tilde{x}_{2}+w_{3}x_{3}+w_{4}x_{4}=d, 
\end{equation}
where the new variable $\tilde{x}_{2}=x_{2}+a_{2}\geq a_{2}$. Now, we remember that Equation (3.1.27) has five solutions, which were obtained in the lemma above:
$$(a_{0},0,0,0,1), (1,a_{1},0,0,0), (0,1,a_{2},0,0), (0,0,1,a_{3},0) \mbox{ and } (0,0,0,1,a_{4}).$$
Since $\tilde{x}_{2}\geq a_{2}$, we have a unique option that solve (3.1.27): $(0,1,a_{2},0,0)$. As a consequence, the equation (3.1.26) has a unique solution. Thus returning to Equation (3.1.26), after subtracting the vector  
$(0,0, a_2, 0, 0)$   
we obtain the solution $(0,1,0,0,0)$ that gives as generator the monomial $z_1.$ Similarly we obtain the other solutions which are exactly the solutions given by Equation (3.1.27). It follows that the generator for 
 $H^{0}(\mathbb{P}(\mathbf{w}),w_{i})$ is  given by the monomial  $z_i$ for $i=0\ldots 4.$ 
\end{proof}
\end{lem}
\medskip

Following \cite{BGN1}, we collect the information given by the previous lemmas

\begin{prop} Let us choose the  definition of the weighted projective space $\mathbb{P}(\mathbf{w})$  as a scheme $\operatorname{Proj}(S(\mathbf{w}))$, where
$$
S(\mathbf{w})=\bigoplus_d S^d(\mathbf{w})=\mathbb{C}\left[z_0, z_1, z_2, z_3, z_4\right].
$$
The ring of polynomials $\mathbb{C}\left[z_0, z_1, z_2, z_3, z_4\right]$ is graded with grading defined by the weights $\mathbf{w}=\left(w_1, w_1, w_2, w_3, w_4 \right)$.  From Lemma 3.1.3, the group $G(\mathbf{w})$ of automorphisms of the graded ring $S(\mathbf{w})$ can be defined on generators by 
$$
\varphi_{\mathrm{w}}\left(\begin{array}{l}
z_0 \\
z_1 \\
z_2 \\
z_3\\
z_4
\end{array}\right)=\left(\begin{array}{l}
\alpha_0 z_0 \\
\alpha_1 z_1 \\
\alpha_2 z_2 \\
\alpha_3 z_3\\
\alpha_4 z_4
\end{array}\right)
$$
where $\alpha_i\in \mathbb C^*.$ 
The group $\mathcal{G}(\mathbf{w})$ of complex automorphisms of $\mathbb{P}(\mathbf{w})$ is the projectivization of  $G(\mathbf{w})$ which in this case is given by  $\mathcal{G}(\mathbf{w})=\left(\mathbb{C}^*\right)^4.$ Actually since the generating set of $H^{0}(\mathbb{P}(\textbf{w}),\mathcal{O}(d))$  is  given by the monomials  $z_0^{a_0}z_4, z_0z_1^{a_1}, z_1z_2^{a_2}, z_2z_3^{a_3}, z_3z_4^{a_4}.$ Thus, the moduli of the orbifold $Z_f$  is included in the space  
$$Span\{z_0^{a_0}z_4, z_0z_1^{a_1}, z_1z_2^{a_2}, z_2z_3^{a_3}, z_3z_4^{a_4}\} / G(\mathbf{w})$$ which determines a 0-dimensional quotient. 
\end{prop}
\hfill$\square$
\smallskip

From the two previous lemmas, we obtain the following outcome in the context of Sasaki-Einstein structures for rational homology 7-spheres.

\begin{prop}
    Let $f$ be a cycle polynomial as in (3.1.1) of degree $d$ with associated weight vector $\textbf{w}=(w_{0},w_{1},w_{2},w_{3},w_{4})$ such that $\gcd(d,w_{i})=1.$ Then the complex dimension  of the  moduli of the orbifold $Z_f=(f=0)/\mathbb{C}^{*}(\textbf{w})$,  is equal to $0$. Moreover,  the generators of the spaces   
$H^{0}(\mathbb{P}(\textbf{w}),\mathcal{O}(d))$  and $H^{0}(\mathbb{P}(\mathbf{w}),\mathcal{O}(w_{i}))$ for all $i=0, \ldots , 4$ are given in Proposition 3.1.1. 
    Additionally, if  $f$ belongs to one of the 236 rational homology spheres admitting Sasaki-Einstein metrics found in \cite{BGN2} and \cite{CL}, then the dimension of the local moduli of Sasaki-Einstein metrics of $\mathbb{Q}$-homology 7-spheres at $L_f$ is zero dimensional.  Thus rational homology 7-spheres given as links coming from polynomials $f$  as above do not admit inequivalent families of Sasaki-Einstein structures.
   \end{prop}
   \hfill$\square$
         
         \smallskip
    \begin{example}     
    Consider the following cycle polynomial $f=z_{4}z_{0}^2+z_{0}z_{1}^{8}+z_{1}z_{2}^{4}+z_{2}z_{3}^{30}+z_{3}z_{4}^{3}$ which can be found in the Johnson and Koll\'ar list of anticanonically embedded Fano K\"ahler-Einsten 3-folds. It follows that the corresponding weight vector is $\textbf{w}=(1945,477,1321,148,1871)$ and the degree is  $d=5761$ so $\gcd(d,w_{i})=1$. As shown in \cite{BGN1} the corresponding link $L_f$ is a Sasaki-Einstein rational homology 7-sphere. By Proposition 3.1.1, the generating set of $H^{0}(\mathbb{P}(\textbf{w}),\mathcal{O}(d))$ is given by the set of monomials $\{z_{4}z_{0}^2, z_{0}z_{1}^{8}, z_{1}z_{2}^{4}, z_{2}z_{3}^{30}, z_{3}z_{4}^{3}\}$. Furthermore, the moduli of the orbifold $Z_{f}$ is included in  
    $$Span\{z_{4}z_{0}^2, z_{0}z_{1}^{8}, z_{1}z_{2}^{4}, z_{2}z_{3}^{30}, z_{3}z_{4}^{3}\}/G(\textbf{w})$$
    where
    $G(\textbf{w})$ is defined on generators by
    $$
\varphi_{\mathrm{w}}\left(\begin{array}{l}
z_0 \\
z_1 \\
z_2 \\
z_3\\
z_4
\end{array}\right)=\left(\begin{array}{l}
\alpha_0 z_0 \\
\alpha_1 z_1 \\
\alpha_2 z_2 \\
\alpha_3 z_3\\
\alpha_4 z_4.
\end{array}\right)
$$
In this case $L_f$ does not admit inequivalent Sasaki-Einstein metrics.
\end{example}

% !TEX root = MainMod.tex
\smallskip

\subsection{Rational homology 7-spheres: Thom-Sebastiani sums of invertible polynomials}
In this subsection we consider the following types of invertible polynomials:
\begin{align*}
    \mbox{Type I (Fermat-Cycle): } & f= z_{0}^{a_{0}}+z_{1}^{a_{1}}+z_{4}z_{2}^{a_{2}}+z_{2}z_{3}^{a_{3}}+z_{3}z_{4}^{a_{4}}\\
    \mbox{Type II (Chain-Cycle): } & f= z_{0}^{a_{0}}+z_{0}z_{1}^{a_{1}}+z_{4}z_{2}^{a_{2}}+z_{2}z_{3}^{a_{3}}+z_{3}z_{4}^{a_{4}}\\
    \mbox{Type III (Cycle-Cycle): } & f= z_{1}z_{0}^{a_{0}}+z_{0}z_{1}^{a_{1}}+z_{4}z_{2}^{a_{2}}+z_{2}z_{3}^{a_{3}}+z_{3}z_{4}^{a_{4}}.
\end{align*}
Let  us assumed that the weight vectors associated to these families of polynomials  have the form:
\begin{equation}
\mathbf{w}=\left(w_0, w_1, w_2, w_3, w_4\right)=\left(m_3 v_0, m_3 v_1, m_2 v_2, m_2 v_3, m_2 v_4\right)
\end{equation}
where $\gcd(m_{2},m_{3})=1$ and degree $d=m_2m_3.$ We will impose the  conditions $\gcd(v_{0},v_{1})=1$ and $\gcd(v_{i},v_{j})=1$ for $i\neq j $ with $i,j\in\{2,3,4\},$ which in particular implies that the weighted hypersurface $Z_f\subset \mathbb{P}(\bf{w})$ is well-formed. 

Before we find the generators of the space of deformations of  the orbifold $Z_f$, we will prove some technical lemmas on the different types of polynomials described above. 
\begin{lem} Let $f$ be an invertible polynomial of type I, II or III with associated weight vector $\bf w$ described in (3.2.1) and with degree $d=m_{2}m_{3}$. Then we have
     \begin{itemize}
         \item[a)] If $f$ is an invertible polynomial of type I, then $v_{0}=v_{1}=1$ in $\bf w$ and hence $w_0=w_1.$ 
         \item[b)]  If $f$ is an invertible polynomial of type II, then $v_{0}=1$ and if the pair $(\bf w, d)$ does not admit a polynomial of type I, then  $v_{1}\neq1$.
         \item[c)] If $f$ is an invertible polynomial of type III, such that $(\bf w, d)$ does not admit a polynomial of type II, then $v_{0}\neq1$ and $v_{1}\neq1$.
     \end{itemize}
     
\begin{proof}
    \begin{itemize}
        \item[a)] For a polynomial $f$ of type I: $$f=z_{0}^{a_{0}}+z_{1}^{a_{1}}+z_{4}z_{2}^{a_{2}}+z_{2}z_{3}^{a_{3}}+z_{3}z_{4}^{a_{4}}$$
        with associated weight vector $\textbf{w}=(w_{0},w_{1},w_{2},w_{3},w_{4})$ and degree $d=m_{2}m_{3}$, we have
        $$
            \left\{
                \begin{array}{rcl}
                    a_{0}w_{0} = d \Rightarrow a_{0}m_{3}v_{0} = m_{2}m_{3}  \\
                    a_{1}w_{1} = d \Rightarrow a_{1}m_{3}v_{1} = m_{2}m_{3}.
                \end{array}
            \right. 
        $$
        The equalities above imply $v_{0}\mid m_{2}$ and $v_{1} \mid m_{2}$. Since $w_{i}=m_{2}v_{i}$ for $i=2,3,4$, we have $ v_{0}\mid\gcd(w_{0},w_{2},w_{3},w_{4})$. Finally, since $\textbf{w}$ is well-formed, we have $v_{0}=1$. A similar argument shows that $v_{1}=1$.
        
        \item[b)] For a polynomial $f$ of type II $$f=z_{0}^{a_{0}}+z_{0}z_{1}^{a_{1}}+z_{4}z_{2}^{a_{2}}+z_{2}z_{3}^{a_{3}}+z_{3}z_{4}^{a_{4}}$$
        with associated weight vector $\textbf{w}=(w_{0},w_{1},w_{2},w_{3},w_{4})$ and degree $d=m_{2}m_{3}$, we have
        $$a_{0}w_{0}=d \Rightarrow a_{0}m_{3}v_{0}=m_{2}m_{3}.$$
        Thus $v_{0} \mid m_{2}$. As $w_{i}=m_{2}v_{i}$ for $i=2,3,4$, we also obtain $v_{0} \mid \gcd(w_{0},w_{2},w_{3},w_{4})$. Since $\textbf{w}$ is well-formed, we conclude that $v_{0}=1$. Now, if we suppose that $v_{1}=1$, then we can choose the polynomial        $$\tilde{f}=z_{0}^{m_{2}}+z_{1}^{m_{2}}+z_{4}z_{2}^{a_{2}}+z_{2}z_{3}^{a_{3}}+z_{3}z_{4}^{a_{4}}$$
        which is of type I for the weight vector $\bf w$, but this contradicts the hypothesis. Thus, we have $v_{1}\neq 1$.
        \item[c)] We consider an invertible polynomial of type III:     $$f=z_{1}z_{0}^{a_{0}}+z_{0}z_{1}^{a_{1}}+z_{4}z_{2}^{a_{2}}+z_{2}z_{3}^{a_{3}}+z_{3}z_{4}^{a_{4}},$$
        which has associated the weight vector $\textbf{w}=(w_{0},w_{1},w_{2},w_{3},w_{4})$ and degree $d=m_{2}m_{3}$. If we suppose that $v_{0}=1$, then $\bf w$ admits a polynomial $\tilde{f}$ of type II:        $$\tilde{f}=z_{0}^{m_{2}}+z_{0}z_{1}^{a_{1}}+z_{4}z_{2}^{a_{2}}+z_{2}z_{3}^{a_{3}}+z_{3}z_{4}^{a_{4}}$$
        which is a contradiction. Thus, we have $v_{0}\neq 1$. A similar process as the one given above leads to $v_{1}\neq 1$.
    \end{itemize}
\end{proof}
\end{lem}

In the next lemma, we show that the existence of certain type of polynomial for a given weight vector $\bf w$ implies the presence of a different type of invertible polynomial associated to the same weight vector $\bf w$.

\begin{lem} Let $f$ be an invertible polynomial of type $I$, $II$ or $III$, where its associated weight vector $\bf w$ is described as in (3.2.1) and its  degree satisfies $d=m_{2}m_{3}$. Then the following hold:
\begin{itemize}
    \item[a)] If $f$ is a polynomial of type II, then its  associated weight vector $\bf w$ also admits a polynomial of type III.
    \item[b)] If $f$ is a polynomial of type I, then its  associated weight vector $\bf w$ also admits polynomials of type II and III.   
\end{itemize}
\begin{proof}
    \begin{itemize}
        \item[(a)]  Let $f$ be a polynomial of type II: $$f=z_{0}^{a_{0}}+z_{0}z_{1}^{a_{1}}+z_{4}z_{2}^{a_{2}}+z_{2}z_{3}^{a_{3}}+z_{3}z_{4}^{a_{4}}.$$
        By Lemma 3.2.1, we know $w_{0}=m_{3}$. As $d=m_{2}m_{3}$, we have $d-w_{1}=m_{3}(m_{2}-v_{1})>m_{3}$. This implies that $m_{3}\mid d-w_{1}$ and there exists an integer $\tilde{a}_{0}>1$ such that 
        $$\tilde{a}_{0}w_{0}=\tilde{a}_{0}m_{3}=d-w_{1}.$$
        Thus, the weights vector $\bf w$ admits a polynomial $\tilde{f}$ of type III:       $$\tilde{f}=z_{1}z_{0}^{\tilde{a}_{0}}+z_{0}z_{1}^{a_{1}}+z_{4}z_{2}^{a_{2}}+z_{2}z_{3}^{a_{3}}+z_{3}z_{4}^{a_{4}}.$$
        
        \item[(b)] If $\bf w$ is the associated weight vector to the invertible polynomial $f$ of type I: $$f=z_{0}^{a_{0}}+z_{1}^{a_{1}}+z_{4}z_{2}^{a_{2}}+z_{2}z_{3}^{a_{3}}+z_{3}z_{4}^{a_{4}},$$
        then by Lemma 3.2.1, we have $w_{0}=w_{1}=m_{3}$. As $d=a_{1}m_{3}$ and $d>w_{0}+w_{1}=2m_{3}$, these imply that $a_{1}\geq3$. Now, if we take  $\tilde{a}_{1}=a_{1}-1\geq2$, we can verify that $w_{0}+\tilde{a}_{1}w_{1}=d$. Therefore, $\bf w$ also admits a polynomial $\tilde{f}$ of type II:       $$\tilde{f}=z_{0}^{a_{0}}+z_{0}z_{1}^{\tilde{a}_{1}}+z_{4}z_{2}^{a_{2}}+z_{2}z_{3}^{a_{3}}+z_{3}z_{4}^{a_{4}}.$$
        Finally, from (a) of this lemma, we have that $\bf w$ also admits a polynomial of type III.
    \end{itemize}
\end{proof}
\end{lem}

For any invertible polynomial $f$ of type I, II or III, whose associated weight vector $\bf w$ is defined as in (3.2.1) and with degree  $d=m_{2}m_{3}$, we have that $f$ contains no block of the form $z_{i}^2+z_{j}^2$ (recall $\gcd(m_2, m_3)=1$). This implies that $\dim \mathfrak{Aut}(Z_f)=0$, see \cite{BGK}. Thus, for the hypersurface $\{f=0\}\subset\mathbb{C}^{5}$,  the complex dimension of the orbifold $Z_f\subset \mathbb{P}(\bf w)$ is given by the formula
\begin{equation}
   \dim_{\mathbb C } H^{0}(\mathbb{P}(\textbf{w}),\mathcal{O}(d))-\sum_{i}\dim_{\mathbb C } H^{0}(\mathbb{P}(\mathbf{w}),\mathcal{O}(w_{i})).
\end{equation}
  
We will begin computing the generators of $H^{0}(\mathbb{P}(\textbf{w}),\mathcal{O}(d))$. As before, we look for all different  monomials $z_{0}^{y_{0}}z_{1}^{y_{1}}z_{2}^{x_{2}}z_{3}^{x_{3}}z_{4}^{x_{4}}$ of degree $d$. Since $\textbf{w}=(w_{0},w_{1},w_{2},w_{3},w_{4})$ is the associated weight vector to $f$, equivalently we can solve the following Diophantine equation:
\begin{equation}
    w_{0}y_{0}+w_{1}y_{1}+w_{2}x_{2}+w_{3}x_{3}+w_{4}x_{4}=d
\end{equation}
with variables $y_{i},x_{i}\in\mathbb{Z}^{+}_{0}$ and where at least one of them is nonzero. The following lemma allows us to split  Equation (3.2.3) in two new Diophantine equations. 

\begin{lem}
    In the Diophantine equation (3.2.3), if some $y_{i}$ is a positive integer, then $x_{j}=0$, for all $j=2,3,4$. The converse is also true.
    \begin{proof}
         On the contrary, let us assume that there exist some positive integers $y_{i}$ and $x_{j}$. In this case, we can suppose without loss of generality that $y_{0}>0$ and $x_{2}>0$. Since $\bf w$ is defined as in (3.2.1), we can write Equation (3.2.3) as:
        \begin{equation*}        m_{3}v_{0}y_{0}+m_{3}v_{1}y_{1}+m_{2}v_{2}x_{2}+m_{2}v_{3}x_{3}+m_{2}v_{4}x_{4}=m_{2}m_{3}.
        \end{equation*}
    Then, we have
    \begin{equation*}
        m_{2}(v_{2}x_{2}+v_{3}x_{3}+v_{4}x_{4})=m_{3}(m_{2}-v_{0}y_{0}-v_{1}y_{1}).
    \end{equation*}
    Since $\gcd(m_{2},m_{3})=1$, then $m_{2}\mid (m_{2}-v_{0}y_{0}-v_{1}y_{1}).$ Moreover, since $m_{2}-v_{0}y_{0}-v_{1}y_{1}<m_{2}$, we obtain $m_{2}-v_{0}y_{0}-v_{1}y_{1}=0.$ Thus, we have   
    $$m_{2}(v_{2}x_{2}+v_{3}x_{3}+v_{4}x_{4})=m_{3}(m_{2}-v_{0}y_{0}-v_{1}y_{1})=0.$$
    Finally, as $x_{2}>0$ and $x_{3},x_{4}\in\mathbb{Z}_{0}^{+}$, the equality above implies that $m_{2}=0$,
    which is not possible. For the converse, the process is similar to the process in the previous argument.
    \end{proof}
\end{lem}
From this lemma we conclude that the solutions of (3.2.3) can be obtained putting together the solutions of each one of the following Diophantine equations:
\begin{equation}
    w_{0}y_{0}+w_{1}y_{1}=d
\end{equation}
and
\begin{equation}
    w_{2}x_{2}+w_{3}x_{3}+w_{4}x_{4}=d.
\end{equation}

Next, we solve Equations (3.2.4) and (3.2.5).  The next lemma will allow us to find all the solutions of Equation (3.2.4).

\begin{lem}
 Let $f$ be an invertible polynomial of type I, II or III, whose associated weight vector is $\textbf{w} =(w_{0},w_{1},w_{2},w_{3},w_{4})$ satisfying (3.2.1) and with degree  $d=m_{2}m_{3}$. Then we have:
 \begin{itemize}
    \item[a)] If $f$ is of type I, then Equation (3.2.4) has $m_{2}+1$ solutions.
    \item[b)]  If $f$ is of type II and its associated weight vector $\bf w$ does not admit polynomial of type I, then Equation (3.2.4) has $ \dfrac{m_{2}}{v_{1}}-\dfrac{1}{v_{1}}+1$  solutions.
    \item[c)]  If $f$ is of type III and its associated weight vector $\bf w$ does not admit polynomial of type II, then Equation (3.2.4) has $\dfrac{m_{2}}{v_{0}v_{1}}-\dfrac{1}{v_{0}}-\dfrac{1}{v_{1}}+1$ solutions.
\end{itemize}
Moreover, all these quantities which represent the number of solutions of (3.2.4) are equivalent to
$$\left\lfloor \dfrac{m_{2}}{v_{0}v_{1}}\right\rfloor+1.$$
  \begin{proof} Indeed  
  \begin{itemize}
          \item[a)] Since $f$ is of type I, we have $v_{0}=v_{1}=1$. These imply that $w_{0}=w_{1}=m_{3}$. Replacing in Equation (3.2.4), we obtain
          $$m_{3}y_{0}+m_{3}y_{1}=d =m_{2}m_{3}.$$
          Simplifying, we obtain a new  Diophantine equation $y_{0}+y_{1}=m_{2}$. Since all solutions are in $\mathbb{Z}_{0}^{+}$, 
          these are given by the pairs
          $$\{ (0,m_{2}), (1,m_{2}-1),\dots (m_{2},0)\}.$$
          Thus Equation (3.2.4) has $m_{2}+1$ solutions.
          \item[b)]  By Lemma 3.2.1, we have $v_{0}=1$ and $v_{1}\neq1$. As a consequence, we obtain $w_{0}=m_{3}$. Moreover, as $d=m_{2}m_{3}$ and $w_{1}=m_{3}v_{1}$, the Equation (3.2.4) can be reduced to
          \begin{equation}
              y_{0}+v_{1}y_{1}=m_{2}.
          \end{equation}
          Since $(m_{2},0)$ is a particular solution of (3.2.6), by Lemma 3.1.1 we have that all its solutions are given by the pairs
          $$(y_{0},y_{1})=(m_{2}-kv_{1},k) \ \ \mbox{ where }k\in\mathbb{Z}.$$
          Since $y_{0},y_{1}\in\mathbb{Z}_{0}^{+}$, we obtain
          \begin{equation}
          0\leq k\leq \dfrac{m_{2}}{v_{1}}.
          \end{equation}
          Moreover, as $\bf w$ admits a polynomial of type II, we have $w_{0}+a_{1}w_{1}=d$. By above, it implies that $1+a_{1}v_{1}=m_{2}$. Thus, we  conclude that $\dfrac{m_{2}-1}{v_{1}}$ is an integer. 
          Considering this in (3.2.7), we write
          $$0\leq k\leq \left(\dfrac{m_{2}}{v_{1}}-\dfrac{1}{v_{1}}\right)+\dfrac{1}{v_{1}}.$$
           Moreover, as $v_{1}\neq1$,  we have the number of solutions of (3.2.4):
           $$\left\lfloor \dfrac{m_{2}}{v_{1}} \right\rfloor+1=\dfrac{m_{2}}{v_{1}}-\dfrac{1}{v_{1}}+1.$$
           \item[c)] Let $f$ be the polynomial of type III associated to $\bf w$:       $$f=z_{1}z_{0}^{a_{0}}+z_{0}z_{1}^{a_{1}}+z_{4}z_{2}^{a_{2}}+z_{2}z_{3}^{a_{3}}+z_{3}z_{4}^{a_{4}}.$$
           Replacing $w_{0}=m_{3}v_{0}$, $w_{1}=m_{3}v_{1}$ and $d=m_{2}m_{3}$ in (3.2.4) and then simplifying, we obtain the new equation
           \begin{equation}
               v_{0}y_{0}+v_{1}y_{1}=m_{2}
           \end{equation}
           Since $\gcd(v_{0},v_{1})=1$ and the pair $(a_{0},1)$ is a solution of (3.2.8), we have that all solutions of the Diophantine equation (3.2.8) are given by the pairs
           $$(y_{0},y_{1})=(a_{0}-kv_{1},1+kv_{0})$$
           As $y_{0},y_{1}\in\mathbb{Z}_{0}^{+}$, then $k$ is restricted to 
           \begin{equation}
               -\dfrac{1}{v_{0}}\leq k\leq \dfrac{a_{0}}{v_{1}} 
           \end{equation}
           Since $\bf w$ does not admit polynomials of type II, we have $v_{0}\neq 1$. Thus, we obtain $k\geq0$. On the other hand, as the polynomial $f$ is associated to $\bf w$, we have $a_{0}w_{0}+w_{1}=d$. This implies that $a_{0}v_{0}+v_{1}=m_{2}$. Then 
           $$\dfrac{a_{0}}{v_{1}}=\dfrac{m_{2}-v_{1}}{v_{0}v_{1}} = \dfrac{m_{2}}{v_{0}v_{1}}-\dfrac{1}{v_{0}} = \left(\dfrac{m_{2}}{v_{0}v_{1}}-\dfrac{1}{v_{0}}-\dfrac{1}{v_{1}}\right) +\dfrac{1}{v_{1}}.$$
           Using again that $\bf w$ does not admit polynomial of type II, we have $v_{1}\neq1$, which implies that $\dfrac{1}{v_{1}}<1$. In addition, since $\alpha(\textbf{w})=\dfrac{m_{2}}{v_{0}v_{1}}-\dfrac{1}{v_{0}}-\dfrac{1}{v_{1}}$ is a integer, from  (3.2.9) we obtain 
           $$0\leq k \leq \left\lfloor\dfrac{a_{0}}{v_{1}}\right\rfloor = \dfrac{m_{2}}{v_{0}v_{1}}-\dfrac{1}{v_{0}}-\dfrac{1}{v_{1}}$$
           Finally, as $\gcd(v_{0},v_{1})=1$, then $0<\dfrac{1}{v_{0}}+\dfrac{1}{v_{1}}$. From $\gcd(v_0, v_1)=1,$ it follows  that the number of solutions of (3.2.4) is $$\left\lfloor\dfrac{a_{0}}{v_{1}}\right\rfloor +1 = \dfrac{m_{2}}{v_{0}v_{1}}-\dfrac{1}{v_{0}}-\dfrac{1}{v_{1}}+1=\left\lfloor\dfrac{m_{2}}{v_{0}v_{1}}\right\rfloor +1.$$ 
           \end{itemize}
  \end{proof}  
\end{lem}

\begin{rem} We have the following  remarks.
    \begin{itemize}
        \item[a)] Notice that if $f$ is a polynomial of type II such that its associated weight vector $\bf w$ admits a polynomial of type I, then we can work as in a) of the  lemma above. Therefore, the equation (3.2.4) has $m_{2}+1$ solutions.
        \item[b)] On the other hand, if $f$ is a polynomial of type III whose associated weight vector $\bf w$ admits a polynomial of type II but not type I, then we can use the case b) of the previous lemma.
        \item[c)] Finally, if $f$ is a polynomial of type III whose associated weight vector admits a polynomial of type I and II, then the number of solutions of (3.2.4)  is obtained as in a) of the previous lemma. 
    \end{itemize}
\end{rem}

It remains to  compute the number of solutions of Equation (3.2.5). Due to the fact that the polynomials of type I, II or III have the same cycle block:
\begin{equation}
z_{4}z_{2}^{a_{2}}+z_{2}z_{3}^{a_{3}}+z_{3}z_{4}^{a_{4}},
\end{equation}
the  number of solutions of the Diophantine Equation (3.2.5) is independent of any type. On the other hand, as the polynomial $f$ has degree $d=m_{2}m_{3}$ and $\bf w$ is written as in (3.2.1), then we obtain from (3.2.10) the following equations for the weights $w_{i}$'s:
\begin{align}
    a_{2}w_{2}+w_{4}=d & \Rightarrow v_{4}+a_{2}v_{2}=m_{3}\\
    a_{3}w_{3}+w_{2}=d & \Rightarrow v_{2}+a_{3}v_{3}=m_{3}\\
    a_{4}w_{4}+w_{3}=d & \Rightarrow v_{3}+a_{4}v_{4}=m_{3}.
\end{align}
In addition, as the link $L_{f}$ is a rational homology sphere, it  verifies Equality (2.1.4): $m_{3}=a_{2}a_{3}a_{4}+1$. Thus, we can write $v_{2}$, $v_{3}$ and $v_{4}$ as
\begin{equation}
    v_{2} = a_{4}a_{3}-a_{3}+1, \ \ v_{3} = a_{2}a_{4}-a_{4}+1, \ \ \mbox{ and } \ v_{4} = a_{3}a_{2}-a_{2}+1.
\end{equation}
Next, we will solve Equation (3.2.5):
\begin{lem}
    The Diophantine equation
    $$w_{2}x_{2}+w_{3}x_{3}+w_{4}x_{4}=d,$$
where $w_{i}$'s are defined as in (3.2.1), has only three solutions.
\begin{proof}
    Since $w_{j}=m_{2}v_{j}$, for $j=2,3,4$ and $d=m_{2}m_{3}$, we have an equivalent equation to (3.2.5):
    \begin{equation}
        v_{2}x_{2}+v_{3}x_{3}+v_{4}x_{4}=m_{3}.
    \end{equation}
As $\gcd(v_{2},v_{3})=1$, we can define a new variable $\tilde{x}=v_{2}x_{2}+v_{3}x_{3}.$ Thus, we can write Equation (3.2.15) as
\begin{equation}
    \tilde{x}+v_{4}x_{4}=m_{3}.
\end{equation}
By (3.2.13), a particular solution of the equation (3.2.16) is given by the pair $(v_{3},a_{4}).$ Then the general solution of (3.2.16) is 
\begin{equation}
    (\tilde{x},x_{4})=(v_{3}+tv_{4},a_{4}-t), \ \ \mbox{ where }t\in\mathbb{Z}.
\end{equation}
Now, we consider the Diophantine equation 
\begin{equation}
 v_{2}x_{2}+v_{3}x_{3}=\tilde{x}=v_{3}+tv_{4}   
\end{equation}
From (3.2.11) and (3.2.12), we have
$$v_{4}=m_{3}-a_{2}v_{2}=(v_{2}+a_{3}v_{3})-a_{2}v_{2}=a_{3}v_{3}+(1-a_{2})v_{2}.$$
Therefore, a particular solution of (3.2.18) is the pair $(t(1-a_{2}),1+ta_{3})$. Then the general solution of (3.2.18) is given by
\begin{equation}
    (x_{2},x_{3})=(t(1-a_{2})+sv_{3},1+ta_{3}-sv_{2}), \ \ \mbox{ where }t,s\in\mathbb{Z}.
\end{equation}
From (3.2.17) and (3.2.19), we obtain the general solution of the Diophantine equation (3.2.15):
\begin{equation}
    (x_{2},x_{3},x_{4}) = (t(1-a_{2})+sv_{3}, 1+ta_{3}-sv_{2},a_{4}-t), \ \ \mbox{ where }t,s\in\mathbb{Z}.
\end{equation}

\underline{\textbf{Claim:}} $\bf s\in\{0,1\}$ 

Since we require $x_{2},x_{3}\in\mathbb{Z}_{0}^{+}$, we have that
$x_{2}=t(1-a_{2})+sv_{3}\geq0$ and $x_{3}=1+ta_{3}-sv_{2}\geq0$. This inequalities imply that
$$\dfrac{sv_{2}-1}{a_{3}}\leq t \leq \dfrac{sv_{3}}{a_{2}-1}.$$
From this, we obtain
\begin{align*}
   \dfrac{sv_{2}-1}{a_{3}}\leq\dfrac{sv_{3}}{a_{2}-1} & \Rightarrow (a_{2}-1)(sv_{2}-1) \leq sv_{3}a_{3} \\
   & \Rightarrow -(a_{2}-1) \leq  s(v_{3}a_{3}-v_{2}(a_{2}-1)).
\end{align*}
Using (3.2.11) and (3.2.12) we have $v_{4}=v_{3}a_{3}-v_{2}(a_{2}-1).$ Replacing above, we can write  
$$-(a_{2}-1) \leq sv_{4} \Rightarrow -\dfrac{a_{2}-1}{v_{4}}\leq s.$$
Now, using the expression for $v_{4}$ given in (3.2.14) and the fact of that $a_{2}a_{3}\geq 2a_{2} >2a_{2}-1,$ we have
$$v_{4}=a_{2}a_{3}-a_{2}+1 >a_{2} \Rightarrow \dfrac{a_{2}-1}{v_{4}}<1.$$ 
Hence 
\begin{equation}
   -1 < -\dfrac{a_{2}-1}{v_{4}}\leq s. 
\end{equation}
On the other hand, as $x_{3}\geq0$ and $x_{4}\geq0$, then $1+ta_{3}\geq sv_{2}$ and $a_{4}\geq t$, respectively. Putting together these two inequalities, we obtain
$$1+a_{4}a_{3}\geq 1+ta_{3} \geq sv_{2}.$$ 
Also, from the expression given for  $v_{2}$ given in (3.2.14) and the inequality $a_{4}a_{3}\geq 2a_{3} > 2a_{3} -1$, we have 
$$2v_{2}=2a_{4}a_{3}-2a_{3}+2 > 1+a_{4}a_{3}.$$
Since $sv_{2}\leq 1+a_{4}a_{3}$, we conclude that
\begin{equation}
    s\leq \dfrac{1+a_{4}a_{3}}{v_{2}}<2.
\end{equation}
From (3.2.21) and (3.2.22), we conclude that  $s\in\{0,1\}.$

Next, we will exhibit the solutions that are obtained for each $s$. 

\underline{ \textbf{For } $\bf s=0$}: In this case, the general solution of (3.2.15) is given by 
$$(x_{2},x_{3},x_{4})=(t(1-a_{2}),1+ta_{3},a_{4}-t).$$
Since $x_{i}\geq0$, we have $t(1-a_{2})\geq0$, $1+ta_{3}\geq0$ and $a_{4}-t\geq0$. As $a_{2}-1>0$, then $t\leq0$. On the other hand, the inequality $1+ta_{3}\geq0$ implies that $t\geq -\frac{1}{a_{3}}>-1$ so $t=0$. Therefore,  $(0,1,a_{4})$ is the unique solution for Equation (3.2.15).

\underline{ \textbf{For } $\bf s=1$}: The general solution of (3.2.15) is given by
$$(x_{2},x_{3},x_{4})=(t(1-a_{2})+v_{3},1+ta_{3}-v_{2},a_{4}-t).$$
Since $x_{4}=a_{4}-t\geq0$, we have $a_{4}\geq t$. On the other hand, as $x_{3}=1+ta_{3}-v_{2}\geq0$ and  $v_{2}=a_{4}a_{3}-a_{3}+1$ in (3.2.14), we obtain
$$t\geq \dfrac{v_{2}-1}{a_{3}}=a_{4}-1$$
Thus, $t\in\{a_{4}-1,a_{4}\}$. So the two solutions for $s=1$ are 
$$\left\{
                \begin{array}{cl}
                    (1,a_{3},0), & \hbox{ if }  t=a_{4} , \\
                    (a_{2},0,1), & \hbox{ if }  t=a_{4}-1.
                \end{array}
            \right. 
        $$
\end{proof}
\end{lem}

Thus, for $f$ be an invertible polynomial with associated weight vector  $\textbf{w} =(w_{0},w_{1},w_{2},w_{3},w_{4})$ satisfying (3.2.1) and  with  degree  $d=m_{2}m_{3}$ we have
\begin{itemize}
\item If $f$ is a polynomial of type I, then the set of generators of 
            $H^{0}(\mathbb{P}(\textbf{w}),\mathcal{O}(d))$ is given by the following monomials of degree $d$:
            $$\{  z_{1}^{m_{2}}, z_{0}z_{1}^{m_{2}-1},\dots , z_{0}^{m_{2}-1}z_{1},z_{0}^{m_{2}}, z_{4}z_{2}^{a_{2}},  z_{2}z_{3}^{a_{3}},  z_{3}z_{4}^{a_{4}} \}.$$ 
  \item If $f$ is a polynomial of type II and its associated weight vector $\bf w$ does not admit polynomial of type I, then the set of generators of  
  $H^{0}(\mathbb{P}(\textbf{w}),\mathcal{O}(d))$ is given by the following monomials of degree $d:$ 
            $$\left \{z_{0}^{m_{2}-kv_{1}}z_{1}^{k},   z_{4}z_{2}^{a_{2}},   z_{2}z_{3}^{a_{3}},   z_{3}z_{4}^{a_{4}},\mbox{ where }  0\leq k  \leq \left\lfloor \dfrac{m_{2}}{v_{1}}\right \rfloor\right \}.$$           
 \item  If $f$ is a polynomial of type III and its associated weight vector $\bf w$ does not admit polynomial of type II, then the set of generators of  
 $H^{0}(\mathbb{P}(\textbf{w}),\mathcal{O}(d))$ is given by the following monomials of degree $d:$ 
            $$\left \{z_{0}^{a_0-kv_{1}}z_{1}^{1+kv_0},   z_{4}z_{2}^{a_{2}},   z_{2}z_{3}^{a_{3}},   z_{3}z_{4}^{a_{4}}, \mbox{ where }  0\leq k  \leq \left\lfloor \dfrac{m_{2}}{v_{0}v_{1}}\right \rfloor\right \}.$$             
            
\end{itemize}
\smallskip

Now  we will compute  
$\dim_\mathbb{C} H^{0}(\mathbb{P}(\textbf{w}),\mathcal{O}(w_{i}))$ for each $w_{i}$, that is, we will  find all the solutions of the following Diophantine equation
\begin{equation}    w_{0}y_{0}+w_{1}y_{1}+w_{2}x_{2}+w_{3}x_{3}+w_{4}x_{4}=w_{i},
\end{equation}
with $x_{j},y_{k}\in\mathbb{Z}_{0}^{+}$, where at least one of them is nonzero. In the next lemma, we will do this for either  $w_{i}=w_{0}$ or $w_{i}=w_{1}.$
\begin{lem}
 Let $f$ be an invertible polynomial of type I, II or III, with associated weight vector  $\textbf{w} =(w_{0},w_{1},w_{2},w_{3},w_{4})$ satisfying (3.2.1) and  with  degree  $d=m_{2}m_{3}$.  
 We consider the following Diophantine equations
   \begin{align}  w_{0}y_{0}+w_{1}y_{1}+w_{2}x_{2}+w_{3}x_{3}+w_{4}x_{4} & =w_{0}\\
   w_{0}y_{0}+w_{1}y_{1}+w_{2}x_{2}+w_{3}x_{3}+w_{4}x_{4} & =w_{1}.
   \end{align}
Then, we have:
\begin{itemize}
    \item[a)]  If $f$ is of type I, then  equations (3.2.24) and (3.2.25) both have two solutions. Moreover,  the set of  generators of $H^{0}(\mathbb{P}(\textbf{w}),\mathcal{O}(w_{i}))$ is  given by 
    $\{z_0, z_1\}$ for  $i=0,1.$
    \item[b)]  If $f$ is of type II and its associated weight vector $\bf w$ does not admit polynomial of type I, then  equations (3.2.24) and (3.2.25) have one  and two solutions, respectively. Moreover, the set of generators of  
    $H^{0}(\mathbb{P}(\textbf{w}),\mathcal{O}(w_{0}))$ is  given by $\{z_0\}$ and the set generators for $H^{0}(\mathbb{P}(\textbf{w}),\mathcal{O}(w_{1}))$  is  
    $\{z_0^{v_1}, z_1 \}.$
    \item[c)]   If $f$ is of type III and its associated weight vector $\bf w$ does not admit polynomial of type II, then the equations (3.2.24) and (3.2.25) both have one  solution. 
    Moreover, the set of generators for 
    $H^{0}(\mathbb{P}(\textbf{w}),\mathcal{O}(w_{0}))$ is  $\{z_0 \}$  and the set of generators for  $H^{0}(\mathbb{P}(\textbf{w}), \mathcal{O}(w_{1}))$ is  $\{z_1\}$
\end{itemize}
\begin{proof}
Let us study each case:
    \begin{itemize}
        \item[(a)] If $f$ is of type I, we can write  
        $$\textbf{w}= (w_{0},w_{1},w_{2},w_{3},w_{4})=(m_{3},m_{3},m_{2}v_{2},m_{2}v_{3},m_{2}v_{4}).$$
        Since $w_{0}=m_{3}$ and $d=m_{2}m_{3}$, adding $(m_{2}-1)w_{0}$ on both sides of (3.2.24), we obtain
        \begin{equation}   
        w_{0}\hat{y}_{0}+w_{1}y_{1}+w_{2}x_{2}+w_{3}x_{3}+w_{4}x_{4}  =d,
        \end{equation}
        where the new variable $\hat{y}_{0}=y_{0}+(m_{2}-1)\geq m_{2}-1$. 
        From Lemmas 3.2.4 and 3.2.5, it follows that  the solutions of Equation (3.2.26) belong to the set 
$$\,\qquad\{(t,m_{2}-t,0,0,0), (0,0,a_{2},0,1), (0,0,1,a_{3},0),(0,0,0,1,a_{4}), \mbox{ where }t\in\mathbb{Z}_{0}^{+}, t\leq m_{2}\}.$$         
Since $\hat{y}_{0}\geq m_{2}-1$, we will find only two possible solutions for (3.2.26): if $\hat{y}_{0}=m_{2}$, we have $y_{0}=1$ and $y_{1}=0$. On the other hand, if $\hat{y}_{0}=m_{2}-1$, we obtain $y_{0}=0$ and $y_{1}=1$. Thus, the solutions obtained are
        $$(1,0,0,0,0) \mbox{ and } (0,1,0,0,0).$$
        Since $w_{1}=w_{0}$, we also obtain two solutions for Equation (3.2.25).
        
        \item[(b)]  First, we will solve Equation (3.2.24). Since $f=z_{0}^{a_{0}}+z_{0}z_{1}^{a_{1}}+z_{4}z_{2}^{a_{2}}+z_{2}z_{3}^{a_{3}}+z_{3}z_{4}^{a_{4}}$ is a polynomial of type II and its associated weight vector $\bf w$ does not admit polynomials of type I, we can express $\bf w$ as
        $$\textbf{w}= (w_{0},w_{1},w_{2},w_{3},w_{4})=(m_{3},m_{3}v_{1},m_{2}v_{2},m_{2}v_{3},m_{2}v_{4}), \ \ \mbox{ where } v_{1}\neq 1.$$
        As $w_{0}=m_{3}$ and $d=m_{2}m_{3}$, then adding $(m_{2}-1)w_{0}$ on both sides of the Equation (3.2.24), we obtain
        \begin{equation}   w_{0}\hat{y}_{0}+w_{1}y_{1}+w_{2}x_{2}+w_{3}x_{3}+w_{4}x_{4}  =d
        \end{equation}
        with new variable $\hat{y}_{0}=y_{0}+(m_{2}-1)\geq m_{2}-1$. By Lemmas 3.2.4 and 3.2.5, the solutions of (3.2.27) belong to the set
        $$\qquad\left\{(m_{2}-tv_{1},t,0,0,0), (0,0,a_{2},0,1), (0,0,1,a_{3},0),(0,0,0,1,a_{4}), \mbox{ where }t\in\mathbb{Z}_{0}^{+}, t\leq \left\lfloor\dfrac{m_{2}}{v_{1}}\right\rfloor \right\}$$
        Notice that this is possible only when $\hat{y}_{0}=m_{2}-tv_{1}\geq m_{2}-1$. Since $v_{1}\neq1$ then $t\leq \dfrac{1}{v_{1}}<1$. Thus, we conclude that $t=0$. In this case, we obtain 
        $\hat{y}_{0}=m_{2}$, which implies $y_{0}=1$ and $y_{1}=0$. Hence, the unique solution for (3.2.24) is $(1,0,0,0,0).$ 
        
On the other hand, to solve Equation (3.2.25), consider the term  $z_{0}z_{1}^{a_{1}}$  of $f$, that leads to  $w_{0}+a_{1}w_{1}=d.$ 
If we add $w_{0}+(a_{1}-1)w_{1}$ to both sides of  Equation (3.2.25), we obtain the new Diophantine equation
        \begin{equation}           
        w_{0}\hat{y}_{0}+w_{1}\hat{y}_{1}+w_{2}x_{2}+w_{3}x_{3}+w_{4}x_{4}=d, 
        \end{equation}
        where $\hat{y}_{0}=y_{0}+1\geq1$ and $\hat{y}_{1}=y_{1}+a_{1}-1\geq a_{1}-1$. By the solutions given above, we can write $\hat{y}_{0}=m_{2}-tv_{1}$ and $\hat{y}_{1}=t$. Thus, we obtain
        $$\dfrac{m_{2}-1}{v_{1}}-1=a_{1}-1\leq t \leq \dfrac{m_{2}-1}{v_{1}}=\left\lfloor \dfrac{m_{2}}{v_{1}}\right\rfloor.$$
        We have
        \begin{itemize}
            \item If $t=\dfrac{m_{2}-1}{v_{1}}-1$, then $\hat{y}_{0}=1+v_{1}$ and $\hat{y}_{1}=\dfrac{m_{2}-1}{v_{1}}-1$. Thus, $y_{0}=v_{1}$ and $y_{1}=0$.
            \item If $t=\dfrac{m_{2}-1}{v_{1}}$, then $\hat{y}_{0}=1$ and $\hat{y}_{1}=\dfrac{m_{2}-1}{v_{1}}$. Thus, $y_{0}=0$ and $y_{1}=1$.
        \end{itemize}
        Thus, Equation (3.2.25) has two solutions.
        
        \item[(c)]  We consider $f$ be an invertible polynomial of type III: 
         $$f=z_{1}z_{0}^{a_{0}}+z_{0}z_{1}^{a_{1}}+z_{4}z_{2}^{a_{2}}+z_{2}z_{3}^{a_{3}}+z_{3}z_{4}^{a_{4}},$$
        whose associated weight vector $\bf w$ does not admit polynomial of type II.
        To solve (3.2.24), we add $a_{1}w_{1}$ to both sides of this equation. Since $w_{0}+a_{1}w_{1}=d$, we obtain
        \begin{equation}
            w_{0}y_{0}+w_{1}\hat{y}_{1}+w_{2}x_{2}+w_{3}x_{3}+w_{4}x_{4}=d
        \end{equation}
        where $\hat{y}_{1}=y_{1}+a_{1}\geq a_{1}$. By Lemmas 3.2.4 c) and 3.2.5, the solutions of (3.2.29) are in the set
        {\small{
        $$\qquad\left\{(a_{0}-tv_{1},1+tv_{0},0,0,0), (0,0,a_{2},0,1), (0,0,1,a_{3},0),(0,0,0,1,a_{4}), \mbox{ where }t\in\mathbb{Z}_{0}^{+}, t\leq \left\lfloor\dfrac{m_{2}}{v_{0}v_{1}}\right\rfloor \right\}$$
        }}
        Then $\hat{y}_{1}=1+tv_{0}\geq a_{1}$. Thus $t\geq \dfrac{a_{1}-1}{v_{0}}$. As $w_{0}+a_{1}w_{1}=d$, we have that $v_{0}+a_{1}v_{1}=m_{2}$. We obtain 
        $$t\geq\dfrac{a_{1}-1}{v_{0}}=\dfrac{m_{2}}{v_{0}v_{1}}-\dfrac{1}{v_{0}}-\dfrac{1}{v_{1}}=\left\lfloor\dfrac{m_{2}}{v_{0}v_{1}}\right\rfloor.$$
        So $t=\left\lfloor\dfrac{m_{2}}{v_{0}v_{1}}\right\rfloor$. Thus, Equation (3.2.29) has only one solution given by $(a_{0}-tv_{1},1+tv_{0},0,0,0).$ 
        Returning to Equation (3.2.24) (after subtracting the vector $(0, a_1, 0,0,0)$)  we obtain the vector $(a_{0}-tv_{1},1+tv_{0}-a_1,0,0,0)$ as solution. In this case, since 
        $t=\dfrac{m_{2}}{v_{0}v_{1}}-\dfrac{1}{v_{0}}-\dfrac{1}{v_{1}},$ it follows that the generator of  $H^{0}(\mathbb{P}(\textbf{w}),\mathcal{O}(w_{0}))$ is given by $\{z_0\}.$ 
        Solutions to Equation (3.2.25) can be found using a similar argument as  the one given above. It follows that the generator of  $H^{0}(\mathbb{P}(\textbf{w}),\mathcal{O}(w_{1}))$ is given by $\{z_1\}.$
    \end{itemize}
\end{proof}
\end{lem}

\begin{rem} We have the following comments.
    \begin{itemize}
        \item[a)] Notice that if $f$ is a polynomial of type II such that its associated weight vector $\bf w$ admits a polynomial of type I, then  part a) of the previous lemma above holds. Thus, 
        Equations (3.2.24) and (3.2.25) have $2$ solutions each one.
        \item[b)] If $f$ is a polynomial of type III whose associated weight vector $\bf w$ admits a polynomial of type II but not type I, then we can use  part b) of the previous lemma above. Thus,  
        Equations (3.2.24) and (3.2.25) have $1$ and $2$ solutions, respectively.
        \item[c)] Finally, if $f$ is a polynomial of type III whose associated weight vector admits polynomial of type I and II, then the number of solutions of the equations can be  obtained as in 
        a) of the previous lemma. That is, Equations (3.2.24) and (3.2.25)  both have  $2$ solutions.
    \end{itemize}
\end{rem}

Next, we will solve Equation (3.2.23) for the remaining cases: $i=2,3$ or $4$.

\begin{lem}
    Let $f$ be a polynomial of type I, II or III and  $\bf w$ the associated weight vector defined as in (3.2.1). Then the equation (3.2.23):    $$w_{0}y_{0}+w_{1}y_{1}+w_{2}x_{2}+w_{3}x_{3}+w_{4}x_{4}=w_{i}$$
    has a unique solution for each $i=2,3,4$. It follows that the set of generators of   $H^{0}(\mathbb{P}(\textbf{w}),\mathcal{O}(w_{i}))$ is given by $\{z_i\}$ for $i=2, 3,4.$
    \begin{proof}
        Since any polynomial $f$ of type I, II or III has the same block of cycle type   $$z_{4}z_{2}^{a_{2}}+z_{2}z_{3}^{a_{3}}+z_{3}z_{4}^{a_{4}},$$
        we can  work with any of these types. On the other hand, by the cyclic form of this block, it is enough to assume $w_{i}=w_{2}$. Thus, Equation (3.2.23) can be written as
        \begin{equation}           w_{0}y_{0}+w_{1}y_{1}+w_{2}x_{2}+w_{3}x_{3}+w_{4}x_{4}=w_{2} 
        \end{equation}
        Adding $a_{3}w_{3}$ to both sides of (3.2.30), we obtain $$w_{0}y_{0}+w_{1}y_{1}+w_{2}x_{2}+w_{3}\hat{x}_{3}+w_{4}x_{4}=d$$
        where  $\hat{x}_{3}=x_{3}+a_{3}\geq a_{3}$. Since $\hat{x}_{3}>0$, the solutions of the above equation are in the set
        $$\left\{ (0,0,a_{2},0,1), (0,0,1,a_{3},0),(0,0,0,1,a_{4}) \right\}.$$
        Since $a_{3}>1$, we have only one solution $(0,0,1,a_{3},0)$. From $\hat{x}_{3}=a_{3}$, we obtain $x_{3}=0$. Then the solution of Equation (3.2.30) is $(0,0,1,0,0)$. Similarly we obtain the solutions 
        $(0,0,0,1,0)$ and $(0,0,0,0,1)$ for $w_3$ and $w_4$ respectively. The last statement of the lemma follows for similar arguments as the ones given in the previous lemmas.
         \end{proof}
\end{lem}
In the same vein as Proposition 3.1.1 we give the following result.  

\begin{prop} The group $G(\mathbf{w})$ of complex  automorphisms of the graded ring $S(\mathbf{w})$ with  
$\operatorname{Proj}(S(\mathbf{w}))$ can be defined on generators for polynomials of type I, II and III thanks  to Lemmas 3.2.6  and  3.2.7 and  we can describe the  moduli of the corresponding orbifold as before. Indeed, let  $\alpha_i, \beta_1\in \mathbb C^*$ and $\mathbb{A}\in GL(2, \mathbb C),$ then we have   
 \begin{itemize}
\item    If $f$ is a polynomial of type I, then $G(\mathbf{w})$ is given on  generators by  
$$
\varphi_{\mathrm{w}}\left(\begin{array}{l}
z_0 \\
z_1 \\
z_2 \\
z_3 \\
z_{4}
\end{array}\right)=\left(\begin{array}{c}
\mathbb{A}\left(\begin{array}{l}
     z_{0}  \\
     z_{1} 
\end{array}\right)
\\
\alpha_2 z_2 \\
\alpha_3 z_3\\
\alpha_4 z_4
\end{array}\right)$$
and  the  moduli of the orbifold $Z_f$  is included in  the quotient
$$Span \left\{z_1^{m_2}, z_0 z_1^{m_2-1}, \ldots, z_0^{m_2-1} z_1, z_0^{m_2}, z_4 z_2^{a_2}, z_2 z_3^{a_3}, z_3 z_4^{a_4}\right\} / G(\mathbf{w}).$$ 

\item If $f$ is a polynomial of type II and its associated weight vector $\bf w$ does not admit polynomial of type I, then $G(\mathbf{w})$ is given on  generators by 
 $$
\varphi_{\mathrm{w}}\left(\begin{array}{l}
z_0 \\
z_1 \\
z_2 \\
z_3 \\
z_{4}
\end{array}\right)=\left(\begin{array}{c}
\alpha_0 z_0  \\
\alpha_1 z_1 +\beta_{1} z_{0}^{v_{1}}\\
\alpha_2 z_2 \\
\alpha_3 z_3\\
\alpha_4 z_4
\end{array}\right)
$$
and  the  moduli of the orbifold $Z_f$   is included in the quotient 
$$Span \left \{z_{0}^{m_{2}-kv_{1}}z_{1}^{k},   z_{4}z_{2}^{a_{2}},   z_{2}z_{3}^{a_{3}},   z_{3}z_{4}^{a_{4}},\mbox{ where }  0\leq k  \leq \left\lfloor \dfrac{m_{2}}{v_{1}}\right \rfloor \right \} / G(\mathbf{w}) .$$

\item Finally, If $f$ is a polynomial of type III and its associated weight vector $\bf w$ does not admit polynomial of type II, then $G(\mathbf{w})$ is given on  generators by 
$$
\varphi_{\mathrm{w}}\left(\begin{array}{l}
z_0 \\
z_1 \\
z_2 \\
z_3 \\
z_{4}
\end{array}\right)=\left(\begin{array}{c}
\alpha_0 z_0  \\
\alpha_1 z_1 \\
\alpha_2 z_2 \\
\alpha_3 z_3\\
\alpha_4 z_4
\end{array}\right)
$$ 
and the  moduli of the orbifold $Z_f$    is included  in the quotient 
$$Span \left \{z_{0}^{a_0-kv_{1}}z_{1}^{1+kv_0},   z_{4}z_{2}^{a_{2}},   z_{2}z_{3}^{a_{3}},   z_{3}z_{4}^{a_{4}}, \mbox{ where }  0\leq k  \leq \left\lfloor \dfrac{m_{2}}{v_{0}v_{1}}\right \rfloor\right \} / G(\mathbf{w}) .$$ 

\end{itemize}
\end{prop}
\hfill$\square$
\smallskip

For the local moduli of Sasaki-Einstein structures for rational homology spheres we can say that this is non-trivial, more precisely we have: 

\begin{prop}
 Let $f$ be an invertible polynomial of type I, II or III, whose associated weight vector is $\textbf{w} =(w_{0},w_{1},w_{2},w_{3},w_{4})$ which is defined as in (3.2.1) with degree  $d=m_{2}m_{3}$. 
 Then the complex dimension $\mu$ of the local moduli of the orbifold $Z_f=\{ f=0\}/\mathbb{C}^{*}(\textbf{w})$ is given as follows
\begin{itemize}
    \item[a)]  If $f$ is of type I, then   $\mu=m_{2}-3$.
    \item[b)]  If $f$ is of type II and its associated weight vector $\bf w$ does not admit a polynomial of type I, then $\mu=\dfrac{m_{2}}{v_{1}}-\dfrac{1}{v_{1}}-2$ .
    \item[c)]   If $f$ is of type III and its associated weight vector $\bf w$ does not admit a polynomial of type II, then $\mu=\dfrac{m_{2}}{v_{0}v_{1}}-\dfrac{1}{v_{0}}-\dfrac{1}{v_{1}}-1$.
\end{itemize}
All these quantities are equivalent to
$$\mu=\dfrac{m_{2}}{v_{0}v_{1}}-\dfrac{1}{v_{0}}-\dfrac{1}{v_{1}}-1.$$ 
Moreover,  the generators of the spaces   
$H^{0}(\mathbb{P}(\textbf{w}),\mathcal{O}(d))$  and $H^{0}(\mathbb{P}(\mathbf{w}),\mathcal{O}(w_{i}))$ for all $i=0, \ldots , 4$ are given in the previous lemmas of this subsection.  Additionally, if $f$ belongs to one of the 236 rational homology spheres admitting Sasaki-Einstein metrics found in \cite{BGN2} and \cite{CL},  then the real dimension of the local moduli of Sasaki-Einstein metrics of rational homology 7-spheres at $L_f$ equals $2\left [ \dfrac{m_{2}}{v_{0}v_{1}}-\dfrac{1}{v_{0}}-\dfrac{1}{v_{1}}-1\right ]$.  
\end{prop}
    \begin{proof} Let us prove these statements case by case:
        \begin{itemize}
            \item[a)] If $f$ is of type I,  from Lemmas 3.2.4 and 3.2.5 we know $\dim_{\mathbb C } H^{0}(\mathbb{P}(\textbf{w}),\mathcal{O}(d))=m_{2}+4$.  Also, from Lemmas 3.2.6 and   3.2.7, we have $$\sum_{i=0}^{4}\dim_{\mathbb C } H^{0}(\mathbb{P}(\textbf{w}),\mathcal{O}(w_{i}))=7.$$
            Thus 
            $$\mu =\dim_{\mathbb C } H^{0}(\mathbb{P}(\textbf{w}),\mathcal{O}(d))-\sum_{i}\dim_{\mathbb C } H^{0}(\mathbb{P}(\mathbf{w}),\mathcal{O}(w_{i}))=m_{2}-3.$$
            
            \item[b)] If $f$ is of type II with associated weight vector $\bf w$  not admitting a polynomial of type I, from Lemmas 3.2.4 and 3.2.5  we know  
            $\dim_{\mathbb C } H^{0}(\mathbb{P}(\textbf{w}),\mathcal{O}(d))=\dfrac{m_{2}}{v_{1}}-\dfrac{1}{v_{1}}+4.$  On the other hand,  Lemmas 3.2.6 and 3.2.7, establish the equality  $$\sum_{i=0}^{4}\dim_{\mathbb C } H^{0}(\mathbb{P}(\textbf{w}),\mathcal{O}(w_{i}))=6.$$
            Thus
            $$\dim_{\mathbb C } H^{0}(\mathbb{P}(\textbf{w}),\mathcal{O}(d))-\sum_{i}\dim_{\mathbb C } H^{0}(\mathbb{P}(\mathbf{w}),\mathcal{O}(w_{i}))=\dfrac{m_{2}}{v_{1}}-\dfrac{1}{v_{1}}-2.$$
            
            \item[c)] If $f$ is of type III and its associated weight vector $\bf w$ does not admit a polynomial of type II, then we obtain from Lemmas 3.2.4 and 3.2.5 that $$\dim_{\mathbb C } H^{0}(\mathbb{P}(\textbf{w}),d)=\dfrac{m_{2}}{v_{0}v_{1}}-\dfrac{1}{v_{0}}-\dfrac{1}{v_{1}}+4.$$ Moreover, from  Lemmas 3.2.6 and 3.2.7, we have $$\sum_{i=0}^{4}\dim_{\mathbb C } H^{0}(\mathbb{P}(\textbf{w}),\mathcal{O}(w_{i}))=5.$$
            Thus, 
            $$\dim_{\mathbb C } H^{0}(\mathbb{P}(\textbf{w}),\mathcal{O}(d))-\sum_{i}\dim_{\mathbb C } H^{0}(\mathbb{P}(\mathbf{w}),\mathcal{O}(w_{i}))=\dfrac{m_{2}}{v_{0}v_{1}}-\dfrac{1}{v_{0}}-\dfrac{1}{v_{1}}-1.$$
            \end{itemize}

As we notice that for any of the cases mentioned above, $\mu$ is equivalent to
        \begin{equation*}
   \dim_{\mathbb C } H^{0}(\mathbb{P}(\textbf{w}),\mathcal{O}(d))-\sum_{i}\dim_{\mathbb C } H^{0}(\mathbb{P}(\mathbf{w}),\mathcal{O}(w_{i}))= \dfrac{m_{2}}{v_{0}v_{1}}-\dfrac{1}{v_{0}}-\dfrac{1}{v_{1}}-1.
\end{equation*}
The last statement then follows from Theorem 2.2.2    

    \end{proof}

\begin{example} Let us consider the following examples, where all the polynomials are  taken from the list of Johnson and Koll\'ar of anticanonically embedded Fano K\"ahler-Einstein 3-folds.

\begin{enumerate}
    \item The polynomial $f=z_{0}^{9}+z_{1}^{9}+z_{4}z_{2}^{2}+z_{2}z_{3}^{2}+z_{3}z_{4}^{19}$ of type I has associated weight vector $\textbf{w}=(77,77,333,180,27)$ and degree $d=m_{3}m_{2}=693$, where $m_{3}=77$ and $m_{2}=9$. The generating set of $H^{0}(\mathbb{P}(\textbf{w}),\mathcal{O}(d))$ is given by the set of thirteen monomials $$\{z_{0}^9, z_{0}^{8}z_{1},\dots, z_{0}z_{1}^{8},z_{1}^{9}, z_{4}z_{2}^{2},  z_{2}z_{3}^{2}, z_{3}z_{4}^{19}\}.$$
    Then the moduli of the orbifold $Z_{f}$  is included in 
    $$Span \left \{z_{0}^9, z_{0}^{8}z_{1},\dots, z_{0}z_{1}^{8},z_{1}^{9}, z_{4}z_{2}^{2},  z_{2}z_{3}^{2}, z_{3}z_{4}^{19}\right \} / G(\mathbf{w}) .$$ 
    where $G(\textbf{w})$ is expressed as above. Thus the complex  dimension of this moduli is six and the corresponding link $L_f$ has a local moduli of Sasaki-Einstein metrics of real dimension twelve.
    
    \item The polynomial $f=z_{0}^{125}+z_{0}z_{1}^{4}+z_{4}z_{2}^{2}+z_{2}z_{3}^{7}+z_{3}z_{4}^{3}$ of type II with degree $d=m_{3}m_{2}=5375$, where $m_{3}=43$ and $m_{2}=125$, and with associated weight vector is $\textbf{w}=(43,1333,1875,500,1625)$. We notice that $\bf w$ does not admit polynomials of type I. Since $v_{1}=31$, then we have that the generating set of $H^{0}(\mathbb{P}(\textbf{w}),\mathcal{O}(d))$ is given by $\left(\left\lfloor \frac{125}{31}\right\rfloor+1\right)+3=8$ monomials. Indeed, the generating set is given by 
    $$\{z_{0}^{125}, z_{0}^{94}z_{1}, z_{0}^{63}z_{1}^{2}, z_{0}^{32}z_{1}^{3}, z_{0}z_{1}^{4}, z_{4}z_{2}^{2}, z_{2}z_{3}^{7}, z_{3}z_{4}^{3}\}$$
    The moduli of the orbifold $Z_{f}$ is obtained in a similar way as above. Thus the complex  dimension of this moduli is two and the corresponding link $L_f$ has a local moduli of Sasaki-Einstein metrics of real dimension four.
    \item The polynomial  $f=z_{1}z_{0}^{5}+z_{0}z_{1}^{15}+z_{4}z_{2}^{2}+z_{2}z_{3}^{4}+z_{3}z_{4}^{4}$ of type III, with associated weight vector $\textbf{w}=(231,66,481,185,259)$ and degree $d=m_{3}m_{2}=1221$, where $m_{3}=33$ and $m_{2}=37$. Here, the weight vector $\bf w$ does not admit polynomials type I or II. Moreover, as $v_{0}=7$ and $v_{1}=2$, then  the generating set of $H^{0}(\mathbb{P}(\textbf{w}),\mathcal{O}(d))$ is formed by $\left(\left\lfloor \frac{37}{(7)(2)}\right\rfloor+1\right)+3=6$ given by 
    $$\{z_{0}^{5}z_{1}, z_{0}^{3}z_{1}^{8}, z_{0}z_{1}^{15}, z_{4}z_{2}^{2}, z_{2}z_{3}^{4}, z_{3}z_{4}^{4}\}.$$ 
It follows that the  the complex  dimension of this moduli is one and the corresponding link $L_f$ has a local moduli of Sasaki-Einstein metrics of real dimension two.
    
\end{enumerate}
\end{example}

\smallskip

\subsection{Rational homology sphere: the Berglund-H\"ubsch dual of  chain-cycle polynomials}

As mentioned before, in \cite{CGL} it was proven that the  Berglund-H\"ubsch transpose rule only produces twins for singularities of cycle type, type I and type III. Actually these types are preserved under the Berglund-H\"ubsch transpose: 

\begin{itemize}
\item Type I  polynomials are sent to type I polynomials.
\item Type III polynomials are sent to type III polynomials, and moreover 
 \item $m_2$, $v_0$ and $v_1$ are invariant under the Berglund-H\"ubsch rule. 
 \end{itemize}
So  the  the moduli of orbifolds $Z_{f^T}$ determined by Berglund-H\"ubsch transpose dual $f^T$ of $f$ for cycle polynomials  and polynomials of type I and type III have moduli described  by Propositions 3.2.2 and 3.1.1. From Proposition 3.2.2 the real dimension  $\mu_{\mathbb R}$ of the local moduli for links arising from polynomials of  type I and type III are given by  $\mu_{\mathbb R}=2[\frac{m_2}{v_0v_1}-\frac{1}{v_0}-\frac{1}{v_1}-1]$, and from Proposition 3.1.2  
$\mu_{\mathbb R}=0$ in case the singularity is given by a cycle type polynomial given in Subsection 3.1.  
Thus we only need to study polynomials of type II, that is polynomials of chain-cycle type, where  the Berglund-H\"ubsch transpose rule does not preserve neither torsion nor Milnor number, and hence $m_2$, $v_0$ and $v_1$ vary. We will also assume that the index $I=|\bf w|-d$ equals 1, as done in \cite{CGL}.
Since our weighted varieties produce Berglund-H\"ubsch duals embedded in non well-formed weighted projective spaces, in principal our procedure will give only upper bounds for the dimension of the moduli, however in most cases we found that these bounds are zero.  
\medskip

Let  
\begin{equation}
z_0^{a_0}+z_0 z_1^{a_1}+z_4 z_2^{a_2}+z_2 z_3^{a_3}+z_3 z_4^{a_4},
\end{equation}
 where its associated weight vector $\bf w$ satisfies condition  (3.2.1). The exponential matrix of $f$ is given by
\begin{equation}
       A_{f}= \begin{bmatrix}
          a_{0} & 0 & 0 & 0 & 0  \\
      1 & a_{1} & 0 & 0 & 0 \\
      0 & 0 & a_{2} & 0 & 1 \\
      0 & 0 & 1 & a_{3} & 0 \\
      0 &  0 & 0 & 1 & a_{4}   
        \end{bmatrix}
        \end{equation}
Applying the Berglund-H\"ubsch transpose rule, one obtains the matrix
\begin{equation*}
       A^{T}_{f}= \begin{bmatrix}
          a_{0} & 1 & 0 & 0 & 0  \\
      0 & a_{1} & 0 & 0 & 0 \\
      0 & 0 & a_{2} & 1 & 0 \\
      0 & 0 & 0 & a_{3} & 1 \\
      0 &  0 & 1 & 0 & a_{4}   
        \end{bmatrix}
        \end{equation*}
which has associated invertible polynomial
\begin{equation}    f^T=z_{0}^{a_{0}}z_{1}+z_{1}^{a_{1}}+z_{3}z_{2}^{a_{2}}+z_{4}z_{3}^{a_{3}}+z_{2}z_{4}^{a_{4}},
\end{equation}
In \cite{CGL}, the associated weight vector $\tilde{\bf w}$ to $f^T$ was obtained: 
$$\tilde{\bf w} = (m_{3}v_{1}(a_{1}-1),m_{3}m_{2}v_{1},m_{2}(m_{2}-1)\tilde{v}_{2},m_{2}(m_{2}-1)\tilde{v}_{3},m_{2}(m_{2}-1)\tilde{v}_{4}),$$
where the degree is $\tilde{d}= m_{3}m_{2}(m_{2}-1)$ and 
\begin{equation}
\tilde{v}_{2}=a_{3}a_{4}-a_{4}+1, \ \ \ \tilde{v}_{3}=a_{2}a_{4}-a_{2}+1 \ \ \mbox{ and } \ \ \tilde{v}_{4}=a_{2}a_{3}-a_{3}+1.
\end{equation}
Also, since $m_{2}-1=a_{1}v_{1}$, we can simplify $\tilde{\bf w}$ and obtain 
\begin{equation}
    \tilde{\bf w} = (m_{3}(a_{1}-1),m_{3}m_{2},m_{2}a_{1}\tilde{v}_{2},m_{2}a_{1}\tilde{v}_{3},m_{2}a_{1}\tilde{v}_{4}),
\end{equation}
where the degree is $\tilde{d}=m_{3}m_{2}a_{1}$.

The complex dimension of the moduli for $Z_{f^T}$ is bounded from below by 
 
\begin{equation}
   \dim_{\mathbb C} H^{0}(\mathbb{P}(\tilde{\bf w}),\mathcal{O}(\tilde{d}))-\sum_{i}\dim_{\mathbb C} H^{0}(\mathbb{P}(\tilde{\bf w}),\mathcal{O}(\tilde{w}_{i}))
\end{equation}
We begin computing $\dim_{\mathbb C} H^{0}(\mathbb{P}(\tilde{\bf w}),\mathcal{O}(\tilde{d}))$. Since this refers to the number of all possible monomials  $z_{0}^{x_{0}}z_{1}^{x_{1}}z_{2}^{x_{2}}z_{3}^{x_{3}}z_{4}^{x_{4}}$ of degree $\tilde{d}$, we calculate this number counting the solutions of the Diophantine equation
\begin{equation}
    \tilde{w}_{0}x_{0}+\tilde{w}_{1}x_{1}+\tilde{w}_{2}x_{2}+\tilde{w}_{3}x_{3}+\tilde{w}_{4}x_{4}=\tilde{d}
\end{equation}
where the unknowns $x_{i}$ are non-negative integers such that at least one of them is non-zero. We will solve this equation adding certain mild constraints on the exponent $a_1$ and $m_3$: either $\gcd(a_{1},m_{3})=1$ and $a_{1}>2$ or $a_{1}=2.$ 
As we will see at the end of this section, the remaining cases of interest can be computed case by case.

\begin{lem}
    Let $f^T$ be an invertible polynomial described as in (3.3.3) with an associated weight vector $\tilde{\bf w}$ given as in (3.3.5). If we add the additional conditions 
    $\gcd(a_{1},m_{3})=1$ and $a_{1}>2$, then the equation (3.3.7) has exactly five solutions. Moreover, the set of generators of  
    $H^{0}(\mathbb{P}(\textbf{w}),\mathcal{O}(\tilde{d}))$ is given by the set $$\{z_0^{m_2}z_1, z_1^{a_1}, z_3z_2^{a_2}, {z_4} z_3^{a_3}, z_2z_4^{a_4}\}.$$
    \begin{proof}
        Using the expression given in (3.3.5), we can write the Diophantine equation (3.3.7) as
        \begin{equation}
            m_{3}(a_{1}-1)x_{0}+m_{3}m_{2}x_{1}+m_{2}a_{1}\tilde{v}_{2}x_{2}+m_{2}a_{1}\tilde{v}_{3}x_{3}+m_{2}a_{1}\tilde{v}_{4}x_{4}=m_{3}m_{2}a_{1},
        \end{equation}
        which we can rewrite as
        $$m_{3}( (a_{1}-1)x_{0}+m_{2}x_{1})=a_{1}m_{2}(m_{3}-\tilde{v}_{2}x_{2}-\tilde{v}_{3}x_{3}-\tilde{v}_{4}x_{4}).$$
        Since $\gcd(m_{2},m_{3})=1$, we obtain $m_{2} \mid (a_{1}-1)x_{0}+m_{2}x_{1}$. This implies that $m_{2} \mid (a_{1}-1)x_{0}$. 
        It is not difficult to obtain  $\gcd(m_{2},a_{1}-1)=1$ (see equation (3.4.11) in the proof of Theorem 3.1 in \cite{CGL}), so $m_{2} \mid x_{0}$. On the other hand, since $m_{3}(a_{1}-1)x_{0} \leq \tilde{d}=m_{3}m_{2}a_{1}$, we have 
        $$x_{0}\leq \dfrac{a_{1}m_{2}}{a_{1}-1}<2m_{2}.$$
         Since $m_{2} \mid x_{0}$, then $x_{0}=0$ or $x_{0}=m_{2}$.
            \begin{itemize}
                \item If $x_{0}=m_{2}$, then the equation (3.3.8) is equivalent to 
                $$m_{3}(a_{1}-1)m_{2}+m_{3}m_{2}x_{1}+m_{2}a_{1}\tilde{v}_{2}x_{2}+m_{2}a_{1}\tilde{v}_{3}x_{3}+m_{2}a_{1}\tilde{v}_{4}x_{4}=m_{3}m_{2}a_{1}$$
            
            Simplifying, we obtain
            $$a_{1}(\tilde{v}_{2}x_{2}+\tilde{v}_{3}x_{3}+\tilde{v}_{4}x_{4})=m_{3}(1-x_{1}).$$
            Since $a_{1}(\tilde{v}_{2}x_{2}+\tilde{v}_{3}x_{3}+\tilde{v}_{4}x_{4})\geq 0$, we have that $x_{1}=1$ or $x_{1}=0$. If $x_{1}=1$, then $\tilde{v}_{2}x_{2}+\tilde{v}_{3}x_{3}+\tilde{v}_{4}x_{4}=0$, which implies that $x_{2}=x_{3}=x_{4}=0$. Thus, we obtain a solution of (3.3.8), which is given by $(m_{2},1,0,0,0)$. On the other hand, if $x_{1}=0$, then $a_{1}(\tilde{v}_{2}x_{2}+\tilde{v}_{3}x_{3}+\tilde{v}_{4}x_{4})=m_{3}$, but $\gcd(a_{1},m_{3})=1$ and $a_{1}\neq 1$. Thus, in this case there is no solution.
            \item If $x_{0}=0$, then the equation (3.3.8) is given by   
            $$m_{3}m_{2}x_{1}+m_{2}a_{1}\tilde{v}_{2}x_{2}+m_{2}a_{1}\tilde{v}_{3}x_{3}+m_{2}a_{1}\tilde{v}_{4}x_{4}=m_{3}m_{2}a_{1}.$$
            Simplifying, we obtain
            $$m_{3}x_{1}=a_{1}(m_{3}-\tilde{v}_{2}x_{2}-\tilde{v}_{3}x_{3}-\tilde{v}_{4}x_{4}).$$
            Since $\gcd(a_{1},m_{3})=1$, then we have $m_{3} \mid (m_{3}-\tilde{v}_{2}x_{2}-\tilde{v}_{3}x_{3}-\tilde{v}_{4}x_{4})$. That is, the expression 
            $\tilde{v}_{2}x_{2}+\tilde{v}_{3}x_{3}+\tilde{v}_{4}x_{4}$ can assume two values: $m_{3}$ or $0$. 
            First, we suppose that $\tilde{v}_{2}x_{2}+\tilde{v}_{3}x_{3}+\tilde{v}_{4}x_{4}=m_{3}$. This forces $x_{1}=0$. Moreover, since  $\tilde{v}_{2}$, $\tilde{v}_{3}$ and $\tilde{v}_{4}$ are given in (3.3.4) and these describe the weight vector $\tilde{\bf w}$  in (3.3.5), then the equation $\tilde{v}_{2}x_{2}+\tilde{v}_{3}x_{3}+\tilde{v}_{4}x_{4}=m_{3}$ has exactly three solutions: $(a_{2},1,0), (0,a_{3},1)$ and $(1,0,a_{4})$. Thus, we obtain three solutions for the equation (3.3.8): 
            $(0,0,a_{2},1,0), (0,0,0,a_{3},1)$ and $(0,0,1,0,a_{4})$. Finally, if we suppose that $\tilde{v}_{2}x_{2}+\tilde{v}_{3}x_{3}+\tilde{v}_{4}x_{4}=0$, then  $x_{1}=a_{1}$. In this case the only solution is $(0,a_{1},0,0,0)$.
            \end{itemize}
            Summarizing, when $a_{1}>2$, we obtain five solutions for Equation (3.3.8):
            $$(m_{2},1,0,0,0), (0,0,a_{2},1,0), (0,0,0,a_{3},1), (0,0,1,0,a_{4}) \mbox { and } (0,a_{1},0,0,0).$$
    \end{proof}
\end{lem}
\begin{prop} Let $f$ be and invertible polynomial as in (3.3.1) with weights satisfying conditions (3.2.1). If its Berglund-H\"ubsch transpose 
     $f^{T}$ with associated weight vector $\bf \tilde w$ satisfies the conditions given in  Lemma 3.3.1, then the complex dimension of the local moduli of the orbifold $Z_{f^T}=\{f^T=0\}/\mathbb{C}^{*}(\bf {\tilde{w}})$ is zero.  Moreover, each $H^{0}(\mathbb{P}(\tilde{\bf w}),\mathcal{O}(\tilde{w}_{i}))$ is generated by $z_i$ for $i=0, \ldots , 4.$
\end{prop}
\begin{proof}
    This result follows from Lemma 3.3.1 and the fact that  $\dim_{\mathbb C} H^{0}(\mathbb{P}(\tilde{\bf w}),\mathcal{O}(\tilde{w}_{i}))\geq 1$  for each $i$ and $\sum_{i=0}^4\dim_{\mathbb C}H^{0}(\mathbb{P}(\tilde{\bf w}),\mathcal{O}(\tilde{w}_{i}))\leq 5.$
\end{proof}
\medskip

\begin{lem}
    Let $f^T$ be an invertible polynomial described in (3.3.3) with associated weight vector $\tilde{\bf w}$ as in (3.3.5). If we add the additional condition $a_{1}=2$, then Equation (3.3.7) has six solutions. 
    Moreover, the set of generators of  
    $H^{0}(\mathbb{P}(\textbf{w}),\mathcal{O}(\tilde{d}))$ is given by the set 
    $$\{z_0^{m_2}z_1, z_0^{2m_2},  z_1^{2}, z_3z_2^{a_2}, z_4{z_3}^{a_3}, z_2z_4^{a_4}\}.$$

    \begin{proof}
    When $a_{1}=2$, Equation (3.3.7) gives 
    \begin{equation}  
    m_{3}x_{0}+m_{3}m_{2}x_{1}+2m_{2}\tilde{v}_{2}x_{2}+2m_{2}\tilde{v}_{3}x_{3}+2m_{2}\tilde{v}_{4}x_{4}=2m_{3}m_{2},
    \end{equation}
    which can be written as
    $$m_{3}(x_{0}+m_{2}x_{1})=2m_{2}(m_{3}-\tilde{v}_{2}x_{2}-\tilde{v}_{3}x_{3}-\tilde{v}_{4}x_{4}).$$
    As $\gcd(m_{2},m_{3})=1$, then $m_{2}\mid (x_{0}+m_{2}x_{1})$. This implies that $m_{2}\mid x_{0}$. Moreover, from (3.3.9) we have $m_{3}x_{0}\leq 2m_{3}m_{2}$. Since $m_{3}\neq 0$, then $x_{0}\leq 2m_{2}$. Thus, $x_{0}$ can take three values: $0$, $m_{2}$ or $2m_{2}$.
    
    Now, we consider two situations: $m_{3}$ is odd or $m_{3}$ is even.
    \begin{itemize}
        \item[a)] When $m_{3}$ is odd. As $\gcd(a_{1},m_{3})=1$, then the method to find the solutions is similar to what we have done for $x_{0}=0$ or $x_{0}=m_{2}$ in the Lemma 3.3.1. For the additional case $x_{0}=2m_{2}$, we have $$2m_{3}m_{2}+m_{3}m_{2}x_{1}+2m_{2}\tilde{v}_{2}x_{2}+2m_{2}\tilde{v}_{3}x_{3}+2m_{2}\tilde{v}_{4}x_{4}=2m_{3}m_{2}$$
            which has solution $x_{1}=x_{2}=x_{3}=x_{4}=0$. So in this case, we add the solution: $(2m_{2},0,0,0,0)$. Thus, we have exactly six solutions.
        \item[b)] When $m_{3}$ is even. Here, we can write $m_{3}=2m_{3}'$. Replacing this in Equation (3.3.9) and simplifying, we obtain
        \begin{equation}  
        m_{3}'x_{0}+m_{3}'m_{2}x_{1}+m_{2}\tilde{v}_{2}x_{2}+m_{2}\tilde{v}_{3}x_{3}+m_{2}\tilde{v}_{4}x_{4}=2m_{3}'m_{2}.
        \end{equation}
        This equation can be written as
        $$m_{3}'(x_{0}+m_{2}x_{1})=m_{2}(2m_{3}'-\tilde{v}_{2}x_{2}-\tilde{v}_{3}x_{3}-\tilde{v}_{4}x_{4}).$$
        By a similar argument used above, we obtain $m_{2} \mid x_{0}$ and $x_{0}\leq 2m_{2}$. This implies that $x_{0}=0$, $x_{0}=m_{2}$ or 
        $x_{0}=2m_{2}$.
        \begin{itemize}
            \item If $x_{0}=0$, then Equation (3.3.10)  can be written as 
            \begin{equation}  
            m_{3}'m_{2}x_{1}+m_{2}\tilde{v}_{2}x_{2}+m_{2}\tilde{v}_{3}x_{3}+m_{2}\tilde{v}_{4}x_{4}=2m_{3}'m_{2}.
            \end{equation}
            Simplifying the expression above, we arrive at the equality
            $$\tilde{v}_{2}x_{2}+\tilde{v}_{3}x_{3}+\tilde{v}_{4}x_{4}=m_{3}'(2-x_{1}).$$
            We notice that $x_{1}$ can assume three values: $x_{1}=0$, $x_{1}=1$ or $x_{1}=2$.
            \begin{itemize}
                \item When $x_{1}=0$, the equation above is
                $$\tilde{v}_{2}x_{2}+\tilde{v}_{3}x_{3}+\tilde{v}_{4}x_{4}=2m_{3}'=m_{3}.$$
                This equation has exactly three solutions: $(a_{2},1,0), (0,a_{3},1)$ and $(1,0,a_{4})$. Thus, the solutions of Equation (3.3.9) are $(0,0,a_{2},1,0), (0,0,0,a_{3},1)$ and $(0,0,1,0,a_{4})$.
                \item When $x_{1}=1$, Equation (3.3.11) can be  simplified: 
                \begin{equation}
                    \tilde{v}_{2}x_{2}+\tilde{v}_{3}x_{3}+\tilde{v}_{4}x_{4}=m_{3}'
                \end{equation}
                If there exists a solution $(\hat{x}_{2},\hat{x}_{3},\hat{x}_{4})$ of this Diophantine equation, then $(2\hat{x}_{2},2\hat{x}_{3},2\hat{x}_{4})$, with $\hat{x}_{i}\in\mathbb{Z}_{0}^{+}$, is solution of the equation
                $$\tilde{v}_{2}x_{2}+\tilde{v}_{3}x_{3}+\tilde{v}_{4}x_{4}=2m_{3}'=m_{3}.$$
                Nevertheless, the solutions of this last equation are $(a_{2},1,0), (0,a_{3},1)$ and $(1,0,a_{4})$, which implies that $2\hat{x}_{i}=1$ for any $i$. That is, $\hat{x}_{i}$ is not integer, which is not possible.
                \item When $x_{1}=2$, the equation is  
                $$\tilde{v}_{2}x_{2}+\tilde{v}_{3}x_{3}+\tilde{v}_{4}x_{4}=0,$$
                which has trivial solution. Therefore, the solution of Equation (3.3.9) is  $(0,2,0,0,0)$.
            \end{itemize}
          \item  If $x_{0}=m_{2}$, then Equation (3.3.10) can be reduced  to
          \begin{equation}
              m_{3}'+m_{3}'x_{1}+\tilde{v}_{2}x_{2}+\tilde{v}_{3}x_{3}+\tilde{v}_{4}x_{4}=2m_{3}'
          \end{equation}
          The above equation can be written as
          $$\tilde{v}_{2}x_{2}+\tilde{v}_{3}x_{3}+\tilde{v}_{4}x_{4}=m_{3}'(1-x_{1}).$$
          Then $x_{1}$ can take two values: $x_{1}=0$ or $x_{1}=1$.
          \begin{itemize}
              \item If $x_{1}=0$, then we obtain Equation (3.3.12), which has no solution. 
              \item If $x_{1}=1$, we obtain the equation 
              $$\tilde{v}_{2}x_{2}+\tilde{v}_{3}x_{3}+\tilde{v}_{4}x_{4}=0,$$
                which has trivial solution. In this case the solution obtained is  $(m_{2},1,0,0,0)$.
          \end{itemize}
          \item If $x_{0}=2m_{2}$. Here, Equation (3.3.10) can be reduced to
          $$m_{3}'m_{2}x_{1}+m_{2}\tilde{v}_{2}x_{2}+m_{2}\tilde{v}_{3}x_{3}+m_{2}\tilde{v}_{4}x_{4}=0$$
          which has trivial solution. In this case, the solution of Equation (3.3.9) is given by $(2m_{2},0,0,0,0)$.       
           \end{itemize}
          In view of the foregoing discussion, we find six solutions for Equation (3.3.9):
          $$ (0,0,a_{2},1,0), (0,0,0,a_{3},1), (0,0,1,0,a_{4}), (0,2,0,0,0), (m_{2},1,0,0,0), \mbox { and } (2m_{2},0,0,0,0).$$
    \end{itemize}
\end{proof}
\end{lem}

\begin{prop}
Let $f$ be and invertible polynomial as in (3.3.1) with weights satisfying conditions (3.2.1). If its Berglund-H\"ubsch transpose 
     $f^{T}$ with associated weight vector $\bf \tilde w$ satisfies the conditions given in  Lemma 3.3.2.  Then the complex dimension of the  moduli of the orbifold $Z_{f^T}=\{f^T=0\}/\mathbb{C}^{*}(\bf {\tilde{w}})$ is zero.  Moreover, $H^{0}(\mathbb{P}(\tilde{\bf w}),\mathcal{O}(\tilde{w}_{1}))$ is generated by $z_0^{m_2}$ and $z_1$  while $H^{0}(\mathbb{P}(\tilde{\bf w}),\mathcal{O}(\tilde{w}_{i}))$ is generated by $z_i,$ for $i\not =1.$  
\end{prop}
\begin{proof}
    First, we will prove that $\dim_{\mathbb C} H^{0}(\mathbb{P}(\tilde{\bf w}),\mathcal{O}(\tilde{w}_{1}))\geq2$. In fact, it is equivalent to show that the quantity of solutions of the following Diophantine equation is not less that two:
    $$\tilde{w}_{0}x_{0}+\tilde{w}_{1}x_{1}+\tilde{w}_{2}x_{2}+\tilde{w}_{3}x_{3}+\tilde{w}_{4}x_{4}=\tilde{w}_{1}.$$  
    Using the expression for $\bf \tilde{w}$ given in (3.3.5) and replacing $a_{1}=2$ in the equation above, we arrive to the equation
    $$m_{3}x_{0}+m_{3}m_{2}x_{1}+2m_{2}\tilde{v}_{2}x_{2}+2m_{2}\tilde{v}_{2}x_{3}+2m_{2}\tilde{v}_{4}x_{4}=m_{3}m_{2}.$$
    Here, we can exhibit at least two solutions: $(m_{2},0,0,0,0)$ and $(0,1,0,0,0)$. Therefore, we obtain $\dim_{\mathbb C} H^{0}(\mathbb{P}(\tilde{\bf w}),\mathcal{O}(\tilde{w}_{1}))\geq2$. 
    
    On the other hand, we know that $\dim_{\mathbb C} H^{0}(\mathbb{P}(\tilde{\bf w}),\mathcal{O}(\tilde{w}_{i}))\geq1$ for $i\neq 1$. This implies that
    $$6\leq\sum_{i=0}^{4}\dim_{\mathbb C} H^{0}(\mathbb{P}(\tilde{\bf w}),\mathcal{O}(\tilde{w}_{i}))\leq \dim_{\mathbb C} H^{0}(\mathbb{P}(\tilde{\bf w}),\mathcal{O}(\tilde{d}))=6.$$
    From here, we can conclude that  $H^{0}(\mathbb{P}(\tilde{\bf w}),\mathcal{O}(\tilde{w}_{1}))$ has dimension 2 and  the set of generators is  $\{z_0^{m_2}, z_1\}$ and each 
    $H^{0}(\mathbb{P}(\tilde{\bf w}),\mathcal{O}(\tilde{w}_{i}))$ is one dimensional and it is generated by $z_i$ for $i\not=1.$

    $$\dim_{\mathbb C} H^{0}(\mathbb{P}(\tilde{\bf w}),\mathcal{O}(\tilde{d}))-\sum_{i=0}^{4}\dim_{\mathbb C} H^{0}(\mathbb{P}(\tilde{\bf w}),\mathcal{O}(\tilde{w}_{i}))=0.$$
\end{proof}

Next we discuss the moduli of the orbifold associated to the polynomials considered in this subsection.

\begin{prop}  For $\mathbb{P}(\mathbf{\tilde{w}})=\operatorname{Proj}(S(\mathbf{\tilde{w}}))$, the  complex  automorphisms of the graded ring $S(\mathbf{\tilde{w}})$ can be defined on generators for the two types of polynomials given in this subsection thanks  to Propositions  3.3.1  and  3.3.2. We can describe the  moduli of the corresponding orbifold, as done previously. Indeed, let  $\alpha_i, \beta_1\in \mathbb C^*,$  then we have   
 \begin{itemize}
\item    If $f$ is a polynomial as in Lemma 3.3.1, then $G(\mathbf{\tilde{w}})$ is given on  generators by  
$$
\varphi_{\mathrm{w}}\left(\begin{array}{l}
z_0 \\
z_1 \\
z_2 \\
z_3 \\
z_{4}
\end{array}\right)=\left(\begin{array}{c}
\alpha_0     z_{0}  \\
\alpha_1     z_{1} \\
\alpha_2 z_2 \\
\alpha_3 z_3\\
\alpha_4 z_4
\end{array}\right)$$
and  the  moduli of the orbifold $Z_{f^T}$  is included in the quotient
$$Span \{z_0^{m_2}z_1, z_1^{a_1}, z_3z_2^{a_2}, {z_4} z_3^{a_3}, z_2z_4^{a_4}\} / G(\mathbf{\tilde{w}})$$ a zero-dimensional quotient space. 

\item    If $f$ is a polynomial as in Lemma 3.3.2, then $G(\mathbf{\tilde{w}}))$ is given on  generators by  
$$
\varphi_{\mathrm{w}}\left(\begin{array}{l}
z_0 \\
z_1 \\
z_2 \\
z_3 \\
z_{4}
\end{array}\right)=\left(\begin{array}{c}
\alpha_0     z_{0}  \\
\alpha_1     z_{1} + \beta_1 z_0^{m_2} \\
\alpha_2 z_2 \\
\alpha_3 z_3\\
\alpha_4 z_4
\end{array}\right)$$
and  the  moduli of the orbifold $Z_{f^T}$  is included in the quotient
$$Span\left\{z_0^{m_2} z_1, z_0^{2 m_2}, z_1^2, z_3 z_2^{a_2}, z_4 z_3^{a_3}, z_2 z_4^{a_4}\right\} / G(\mathbf{\tilde{w}})$$ again  a zero-dimensional quotient space. 
\end{itemize}
\hfill$\square$
\end{prop}
\smallskip

In \cite{CGL}, Theorem 4.1 we established the existence of 75 new rational homology 7-spheres admitting Sasaki-Einstein
metrics with not well-formed quotients, that is, with some group elements having codimension 1 fixed point sets. All these links, which are  listed in a table in the Appendix in \cite{CGL}, are given by polynomials of chain-cycle type of type II with weights satisfying the conditions given in (3.2.1).  So we have: 

\begin{prop}
Let $f$ be and invertible polynomial as in (3.3.1) with weights satisfying conditions (3.2.1). If its Berglund-H\"ubsch transpose 
     $f^{T}$ with associated weight vector $\bf \tilde w$ satisfies the conditions given in either  Lemma 3.3.1 or Lemma 3.3.2 and additionally $f$ belongs to one of the 236 rational homology 7-spheres admitting 
     Sasaki-Einstein metrics found in \cite{BGN2} and \cite{CGL}, then the dimension of the local moduli space of Sasaki-Einstein metrics of the rational homology 7-spheres at  $L_{f^T}$ equals zero. 
     Thus rational homology 7-spheres given as links coming from polynomials $f^T$  as above do not admit inequivalent families of Sasaki-Einstein structures.
     \end{prop}
\begin{proof}
This result follows from Propositions 3.3.1 and 3.3.2 
\end{proof}

\begin{rem}
    From the list of 75 new examples of Sasaki-Einstein rational homology 7-sphere given in the Appendix of \cite{CGL}, there are 7  weight vectors $\bf \tilde{w}$ that do not satisfy the additional conditions of either  Lemma 3.3.1 or  Lemma  3.3.2, actually all of the elements of this table  satisfy $a_1>2$ but $\gcd(m_3, a_1)\not =1.$  For these, we have the following table where $b_{\mu_{\mathbb R}}$ denotes the upper bound for the real dimension of  the local moduli of Sasaki-Einstein metrics of rational homology 7-spheres at $L_{f^T}.$ 
\end{rem}
  
\begin{longtable}{| c | c | c | c | c | c| } \hline
$\bf \tilde{w}$ & $\tilde{d}$ & $m_{3}$ &   $\dim_{\mathbb C} H^{0}(\mathbb{P}(\tilde{\bf w}),\mathcal{O}(\tilde{d}))$   &  $\sum_{i=0}^{4}\dim_{\mathbb C} H^{0}(\mathbb{P}(\tilde{\bf w}),\mathcal{O}(\tilde{w}_{i}))$  & $b_{\mu_{\mathbb R}}$ \\ \hline \hline \endfirsthead
\hline
(177,295,270,370,70)  & 1180 & 118 & 5 & 5  & 0  \\  \hline

(52,663,867,1581,153) & 3315 & 65 & 9 & 6  & 6 \\ \hline

(148,777,987,1911,63) & 3885 & 185 & 9 & 6 & 6  \\ \hline

(86,3655,5185,595,1445) & 10965 & 129 & 6 & 5 & 2 \\ \hline

(86,3655,4165,2635,425) & 10965 & 129 & 6 & 5 & 2 \\ \hline

(438,4161,6175,133,1577) & 12483 & 657 & 6 & 5 & 2 \\ \hline

(438,4161,4693,3097,95) & 12483 & 657 & 6 & 5  & 2 \\ \hline
\end{longtable}

In contrast to the previous cases obtained by the Berglund-H\"ubsch rule,  with the exception of the first member of the above table, all the generating monomials of the local moduli are such that the weight vectors admit blocks of the form $z_{\alpha}^{n_{\alpha}}z_{\beta}^{n_{\beta}}$ or $z_{\alpha}^{n_{\alpha}}z_{\beta}^{n_{\beta}}z_{\lambda}$.  For instance,  the weight vector $\tilde{\textbf{w}}=(52,663,867,1581,153)$ determines as the set of generators of $H^{0}(\mathbb{P}(\tilde{\textbf{w}}),\mathcal{O}(\tilde{d}))$ the following
$$\{ z_{1}z_{0}^{51},z_{1}^{5},z_{1}^{3}z_{2}z_{4}^{3}, z_{1}^{2}z_{4}^{13}, z_{1}z_{2}^{2}z_{4}^{6}, z_{1}z_{3}z_{4}^7, z_{2}^2z_{3},z_{2}z_{4}^{16},z_{3}^2z_{4}\}.$$ The set  $H^{0}(\mathbb{P}(\tilde{\bf w}),\mathcal{O}(\tilde{w}_{i}))$ is spanned by the automorphisms 
$$\varphi_{\tilde{\mathrm{w}}}\left(\begin{array}{l}
z_0 \\
z_1 \\
z_2 \\
z_3 \\
z_{4}
\end{array}\right)=\left(\begin{array}{c}
\alpha_0 z_0  \\
\alpha_1 z_1 \\
\alpha_2 z_2 \\
\alpha_3 z_3+\beta_{3}z_{1}z_{4}^6\\
\alpha_4 z_4
\end{array}\right).
$$ 
Hence  the complex dimension of the moduli of the corresponding orbifold has three as  an upper 
bound. 
\medskip

\section{Application: links of non-isolated singularities and klt singularities}

The explicit description of the monomials generating the vector space $H^{0}(\mathbb P(\mathbf{w}),\mathcal{O}(d))$ yields arguments to  the study  of the  non-quasismooth polynomials generated by  monomials in this set. Then, the  weighted hypersurfaces determined by these sort of  polynomials can be considered as points in the  boundary of a compactification of the moduli of quasismooth polynomials since the  subset of all quasismooth elements is dense in the set of monomials generating  $H^0(\mathbb{P}(\mathbf{w}), \mathcal{O}(d)).$ Below, we explain via examples,  
 how to produce non-quasismooth hypersurfaces with klt singularities which are the  candidates to give rise to non-smooth links whose metric cones can be considered as some sort of degenerating Calabi-Yau cones \cite{LX, Od1}. 
For  precise definitions of the singularities of the minimal model appearing in this section and their possible relations see \cite{Ko3, KM}. In particular, at the end of page 42 in \cite{Ko3}, Kollár gives  conditions needed to ensure that canonical singularities are  klt, these conditions are  trivially  satisfied in our setting.

%Recall \cite{Ko3} that a pair $(X, D)$  consisting of a normal variety $X$ (of complex dimension greater %than 1) and an effective $\mathbb{Q}$-divisor $D$ is said to be klt  if the following are satisifed 
%\begin{itemize}
%\item 
%$K_X+D$ is $\mathbb{Q}$-Cartier. 
%\item For any log resolution $\pi: Y \rightarrow X$ we have
%$$
%K_Y = \pi^*\left(K_X+D\right)+\sum_{E_i} a\left(E_i, X, D\right) E_i,
%$$
%where the $E_i's$ are prime divisors.   
%\item The discrepancy of $E_i$, $a\left(E_i, X, D\right)\in\mathbb{Q}$,  is greater than -1 for all $i.$ 
%\end{itemize}
%If  $a\left(E_i, X, D\right)>-1$  for all $E_i.$ If $a\left(E_i, X, D\right)\geq 0$ for all  $E_i$ that is an exceptional divisor, the pair $(X, D)$ is said to be canonical. A singularity $x \in X$ is klt  if $X$ is klt  around $x$.
\medskip

First,  recall the notion of the Newton polyhedra: let us write the monomial  $x_0^{a_0} \cdots x_n^{a_n}$ as $\mathbf{x}^\mathbf{a}$, where $\mathbf{a}=\left(a_0, \ldots, a_n\right) \in \mathbb{Z}^{n+1}.$ Given a polynomial function 
$\displaystyle{f(x)=\sum_{\mathbf{a} \in \mathbb{Z}_{+}^{n+1}} c_{\mathbf{a}} \mathbf{x}^\mathbf{a}}$, then the support of $f$ is defined as  
$\operatorname{supp}(f):=\left\{\mathbf{a} \in \mathbb{Z}_{+}^{{n+1}} \mid c_{\mathbf{a}} \neq 0\right\}.$ The 
Newton polyhedron $\Gamma_{+}(f)$ of  $f$ is defined to be the convex hull of the following set 
$\displaystyle{\bigcup_{\mathbf{a} \in \operatorname{supp}(f)}\left(\mathbf{a}+\mathbb{R}_{\geq 0}^{n+1}\right)}$. 
For each face $\gamma$ of $\Gamma_{+}(f)$, we define the polynomial $f_\gamma$ as follows:
$$
f_\gamma=\sum_{\mathbf{a} \in \gamma} c_{\mathbf{a}} \mathbf{x}^{\mathbf{a}}.
$$
A power series $f$ is said to be Newton non-degenerate, if for every face $\gamma$ the equation $f_\gamma=0$ defines a hypersurface smooth in the complement of the hypersurface $x_0 \cdots x_n=0$.
\medskip

Now we state the following criterion which we will use below, following \cite{To}, to determine whether the singularity is klt. See \cite{IP} for a proof of the following lemma.
\begin{lem}
    Let $S\subset\mathbb{C}^{n+1}$ be a normal hypersurface defined as the set of zeros of a Newton non-degenerate polynomial $f.$  If the point $(1,1,\dots,1)$ is in the interior of the Newton polyhedron $\Gamma_{+}(f)$, then $S$ has canonical singularities.
    %$\displaystyle f=\sum_{\lambda\in I}c_{\lambda}z^{\lambda}\in\mathbb{C}[z_{0},z_{1},\dots, z_{n}]$, %where $\lambda=(\lambda_{0},\dots,\lambda_{n})$ is a multi-index. If the point $(1,1,\dots,1)$ is in %the interior of the Newton polyhedron $\Gamma_{+}(f)$ generated by the set of multi-index $I$, then %$S$ has canonical singularities.
\end{lem}
\hfill$\square$
\medskip

\begin{itemize}[leftmargin=0pt]
\item Now, consider Example 3.2.1 (1):  the polynomial of type I given by  $$f=z_{0}^{9}+z_{1}^{9}+z_{4}z_{2}^{2}+z_{2}z_{3}^{2}+z_{3}z_{4}^{19},$$  has associated weight vector $\textbf{w}=(77,77,333,180,27)$ and degree $d=693,$ which determines a well-formed Fano  K\"ahler-Einstein 3-fold with at worst cyclic singularities \cite{JK}. The set  $H^{0}(\mathbb{P}(\textbf{w}),\mathcal{O}(d))$ is given by the set of thirteen monomials $$\{z_{0}^9, z_{0}^{8}z_{1},\dots, z_{0}z_{1}^{8},z_{1}^{9}, z_{4}z_{2}^{2},  z_{2}z_{3}^{2}, z_{3}z_{4}^{19}\}.$$
 One can select a certain subset of these monomials and consider  
 the non-quasismooth hypersurface $$X_0:z_{0}^{9}+z_{0}^{2}z_{1}^{7}+z_{4}z_{2}^{2}+z_{2}z_{3}^{2}+z_{3}z_{4}^{19}=0$$ 
 which belongs to the family of weighted hypersurfaces 
 $$X_t:z_{0}^{9}+(1-t)z_{0}^{2}z_{1}^{7}+tz_{1}^9+z_{4}z_{2}^{2}+z_{2}z_{3}^{2}+z_{3}z_{4}^{19}=0\subset \mathbb{P}(77,77,333,180,27)\times \mathbb{C}.$$ From Lemma 2.1.1, it follows  that $X_t$ is quasismooth for all $t\not =0.$  As we show below, the subvariety $X_0$ is klt. Indeed, since the number of monomials defining $X_0$ is equal to the number of variables, any linear combination of these monomials with nonzero coefficients defines a hypersurface isomorphic to $X_0.$ Thus we can take $X_0$ to be a general divisor in the  linear system defined by the monomials of the defining equation.  The base locus of this linear system is contained in the following set of points  
    \begin{align*}
            B=\bigl\{ & [0:0:1:0:0], [0:0:0:1:0], [0:0:0:0:1], [0:1:1:0:0],[0:1:0:1:0],\\ & \hspace{4cm} [0:1:0:0:1]
             , [0:1:0:0:0]  \bigr\}
        \end{align*}
        By Bertini's Theorem on $\mathbb C^5-\{0\}$, we know that $X_0$ is quasismooth outside these points. Moreover, from setting the gradient of $f$ equal to  zero, we find  that $X_0$ is quasismooth at all  elements of $B$ except at the point $[0:1:0:0:0]$. Actually  $X_0$ is klt at $[0:1:0:0:0]$. As mentioned above, a general linear combination of the monomials defining  $X_0$, determines a hypersurface   $$\tilde{X_0}=c_0z_{0}^{9}+c_1z_{0}^{2}z_{1}^{7}+c_2z_{4}z_{2}^{2}+c_3z_{2}z_{3}^{2}+c_4z_{3}z_{4}^{19}=0$$ isomorphic to $X_0$ for ${c_i}'s$ general complex numbers different than zero. So we will show that a general hypersurface $\tilde{X_0}$ is klt at $[0:1:0:0:0].$ In the affine chart $z_1\not =0$ 
        we take $z_{1}=1$ and locally $\tilde{X_0}$ is the quotient of the  hypersurface $$S: c_0z_{0}^{9}+c_1z_{0}^{2}+c_2z_{4}z_{2}^{2}+c_3z_{2}z_{3}^{2}+c_4z_{3}z_{4}^{19}=0$$ in $\mathbb{C}^4$ by the group $\mathbb Z_{77}.$ 
         Since klt is a property  preserved by finite quotients (\cite{Ko3}, Corollary 2.43), it suffices to show that such the general hypersurface $S\subset \mathbb{C}^4$ has canonical singularities.  
Clearly, $S$ is normal (it is a hypersurface and its singular set has codimension at least 2).  
Also, since the coefficients of the monomials in this example are taken to be general, the Newton non-degeneracy is satisfied since the base locus of any collection of the monomials in the equation $c_0z_{0}^{9}+c_1z_{0}^{2}+c_2z_{4}z_{2}^{2}+c_3z_{2}z_{3}^{2}+c_4z_{3}z_{4}^{19}=0$ is contained in the hypersurface $x_0\cdots x_n=0$. So we can apply the criterion given above to determine whether the singularity is klt: the singularity in  $S$ is canonical if the point $(1,1,1,1)$ is in the interior of the Newton polyhedron $\Gamma_{+}(g)$ generated by support of the polynomial $g$ defining $\tilde{X_0}$, that is, generated by the set  
$$\operatorname{Supp}(g)=\bigl\{(9,0,0,0),(2,0,0,0),(0,2,0,1),(0,1,2,0),(0,0,1,19)\bigr\}.$$ 
         We notice that it is enough to show that there exists some point in the Newton polyhedron $\Gamma_{+}(g)$ with all its entries less than 1 . For this, we consider the following point
        $$P_{0}=\dfrac{43}{539}(9,0,0,0)+\dfrac{76}{539}(2,0,0,0)+\dfrac{20}{77}(0,2,0,1)+\dfrac{37}{77}(0,1,2,0)+\dfrac{3}{77}(0,0,1,19)$$
        which has all its entries equal to 1. Thus, the point $(1,1,1,1)$ is in the interior of  $\Gamma_{+}(f).$ We conclude that $X_0$ is klt.
\smallskip

\item Similar arguments can be used to study the non-quasismooth  hypersurfaces for  the two remaining types of polynomials producing  non-trivial moduli. 
For instance, consider  Example 3.2.1 (2): the  polynomial  $$f=z_0^{125}+z_0 z_1^4+z_4 z_2^2+z_2 z_3^7+z_3 z_4^3$$  of type II with weight vector  
$\boldsymbol{w}=(43,1333,1875,500,1625)$ and degree $d=5375$ which determines a well-formed Fano  K\"ahler-Einstein 3-fold with at worst cyclic singularities \cite{JK}. The set   $H^0(\mathbb{P}(\boldsymbol{w}), \mathcal{O}(d))$ is  given by 
$$
\left\{z_0^{125}, z_0^{94} z_1, z_0^{63} z_1^2, z_0^{32} z_1^3, z_0 z_1^4, z_4 z_2^2, z_2 z_3^7, z_3 z_4^3\right\}.
$$
From this set of generators, we consider  the family 
$$X_t: (1-t)z_{0}^{63}z_{1}^{2}+ z_{0}z_{1}^{4} +tz_0^{125}+z_{4}z_{2}^{2}+z_{2}z_{3}^7+z_{3}z_{4}^{3}=0\subset \mathbb{P}(43,1333,1875,500,1625)\times \mathbb{C}$$
with central fiber  $$X_0:z_{0}^{63}z_{1}^{2}+z_{0}z_{1}^{4}+z_{4}z_{2}^{2}+z_{2}z_{3}^7+z_{3}z_{4}^{3}=0,$$  and, using Lemma 2.1.1, with $X_t$ quasismooth for all $t\not=0.$  
One  can verify that  $X_0$ is either smooth or quasismooth in all points  except at the point $[1:0:0:0:0]$. With identical arguments as the ones used above, one can show that the weighted subvariety is klt at the point $[1:0:0:0:0]$.
\medskip

\item Of course, one also can describe non-quasismooth hypersurfaces from orbifolds not belonging to the list of anticanonically embedded K\"ahler-Einstein Fano 3-folds given by Johnson and Koll\'ar \cite{JK}, that is, where the index corresponding to the quasismooth weighted variety has index $I>1.$  
For instance,  consider the polynomial  $$f=z_{0}^{5}+z_{1}^{5}+z_{4}z_{2}^{3}+z_{2}z_{3}^{3}+z_{3}z_{4}^{9}$$ which determines a weighted hypersurface $X$  with  weight vector $\textbf{w}=(82,82,125,95,35)$ and degree $d=410$. We notice that in this case the index is $I=9$. Since $Id=9(410)<\frac{4}{3}(82)(35)$, it follows from  Theorem 2.1.1 that  the weighted hypersurface $X\subset\mathbb{P}(\textbf{w})$ admits a Kähler-Einstein orbifold metric and hence  the link $L_{f}$  admits a Sasaki-Einstein metric as well.  By Proposition 3.2.2, the  set of monomials generating  $H^{0}(\mathbb{P}(\textbf{w}),\mathcal{O}(d))$ is given by
    $$\{z_{0}^5, z_{0}^{4}z_{1},\dots, z_{0}z_{1}^{4},z_{1}^{5}, z_{4}z_{2}^{3},  z_{2}z_{3}^{3}, z_{3}z_{4}^{9}\}.$$ 
One considers  the following family generated by some elements of this set: $$X_t:(1-t) z_0^2 z_1^3+z_0^5+t z_1^5+z_4 z_2^3+z_2 z_3^3+z_3 z_4^9=0 \subset \mathbb{P}(82,82,125,95,35) \times \mathbb{C}$$ 
with central fiber  $$X_{0}:g=z_{0}^{5}+z_{0}^{2}z_{1}^3+z_{4}z_{2}^{3}+z_{2}z_{3}^{3}+z_{3}z_{4}^{9}=0.$$  From Lemma 2.1.1,  $X_t$ is quasismooth for all $t\not=0$ and $X_0$ is non-quasismooth.  As before, it is not difficult to show that  $X_0$ is either smooth or quasismooth in all points  except at the point $[0:1:0:0:0]$.  Actually  $X_0$ is klt at $[0:1:0:0:0]$. Let us consider  a general linear combination  of the monomials defining  $X_0$:  $$\tilde{X_0}=c_0z_{0}^{5}+c_1z_{0}^{2}z_{1}^{3}+c_2z_{4}z_{2}^{3}+c_3z_{2}z_{3}^{3}+c_4z_{3}z_{4}^{9}=0$$  which is  isomorphic to $X_0$ for ${c_i}'s$ general complex numbers different than zero. This  general hypersurface 
$\tilde{X_0}$ is klt at $[0:1:0:0:0].$  Indeed, in the affine chart $z_1\not =0$ we take $z_{1}=1$ and locally $\tilde{X_0}$ is the quotient of the  hypersurface $$S: c_0z_{0}^{5}+c_1z_{0}^{2}+c_2z_{4}z_{2}^{3}+c_3z_{2}z_{3}^{3}+c_4z_{3}z_{4}^{9}=0$$ in $\mathbb{C}^4$ by the group $\mathbb Z_{82}.$ 
         Since klt is a property  preserved by finite quotients, it suffices to show that such the general hypersurface $S\subset \mathbb{C}^4$ has canonical singularities.  
Clearly, $S$ is normal and   since the coefficients of the monomials in this example are taken to be general, the Newton non-degeneracy is satisfied since the base locus of any collection of the monomials in the equation $c_0z_{0}^{5}+c_1z_{0}^{2}+c_2z_{4}z_{2}^{2}+c_3z_{2}z_{3}^{2}+c_4z_{3}z_{4}^{19}=0$ is contained in the hypersurface $x_0\cdots x_n=0$. So we can apply Lemma 4.0.1 to determine whether the singularity is klt: the singularity in  $S$ is canonical if the point $(1,1,1,1)$ is in the interior of the Newton polyhedron $\Gamma_{+}(h)$ generated by support of the polynomial $h$ defining $\tilde{X_0}$, that is, generated by the set  
$$\operatorname{Supp}(h)=\bigl\{(5,0,0,0),(2,0,0,0),(0,3,0,1),(0,1,3,0),(0,0,1,9)\bigr\}.$$ 
         As before,  it is enough to show that there exists some point in the Newton polyhedron $\Gamma_{+}(h)$ with all its entries less than 1 . For this, we consider the following point
        $$P_{0}=\dfrac{44}{205}(3,0,0,0)+\dfrac{73}{410}(2,0,0,0)+\dfrac{19}{82}(0,3,0,1)+\dfrac{25}{82}(0,1,3,0)+\dfrac{7}{82}(0,0,1,9)$$
        which has all its entries equal to 1. Thus, the point $(1,1,1,1)$ is in the interior of  $\Gamma_{+}(h),$ so  $X_0$ is klt.
\medskip

\item In the context of Sasakian geometry, the links determined by the klt Fano  varieties described above are not smooth anymore. For instance  the hypersurface  $$f=z_{0}^{9}+z_{0}^{2}z_{1}^{7}+z_{4}z_{2}^{2}+z_{2}z_{3}^{2}+z_{3}z_{4}^{19}=0$$ on $\mathbb C^5-\{0\}$ has  
a one dimensional singular set   
$\Sigma=\{(0,z_1, 0,0,0)\}$ which intersects the unit sphere $S^9$ in the circle 
 $S^1$ so the link $L_f$ is non-smooth.  
Since the contact 1-form on the link $L_f=V_f\cap S^{9}$ is given by the contact 1-form on the weighted sphere $\eta_{\mathbf{w}}=\frac{\eta}{\sum_{i=0}^4 w_i\left|z_i\right|^2}$ restricted to $L_f,$ one notices that this contact 1-form degenerates on the  singular set $\Sigma\cap S^9$. 
If one considers the open manifold resulting from excluding this singular set we still obtain 
a Reeb vector field $\xi$ on the regular part $L_f^{reg}$ and a Riemannian metric on it. This natural way to produce singular links with Sasaki  metrics on the regular part can be interpreted as the horizon (or base) of a metric cone $C(L_f)=L_f\times \mathbb R^{+}$ with a function $r:C(L_f)\rightarrow {\mathbb R}^{+}$ that determines the radial coordinate and hence a Liouville vector field $\Psi=r\partial_r$ on the regular part  of $C(L_f).$  In accordance to the Yau-Tian-Donaldson conjecture for singular varieties  (\cite{LXZ} Theorem 1.6)   in order to extend  this argument to the realm of Sasaki-Einstein metrics on these singular links and on the corresponding weak Ricci-flat metrics on the associated metric cone one needs to show the exceptionality of the klt singularity (\cite{OS} Theorem 1.4), that is, determine whether the pair ($X, D$) is klt for every effective $\mathbb{Q}$-divisor $D$ that is $\mathbb{Q}$-linearly equivalent to $-K_X$,  or equivalently, show that the $\alpha$-invariant of $X$ is greater than 1 \cite{Bi}. Based on the explicit generators found in this article and the method of the weighted tangent cone developed by Totaro in \cite{To},  we do this with certain generality in \cite{CL2}.  
\end{itemize}

\begin{center}
{\bf Declarations}
\end{center}
%\subsection*{Ethics approval and consent to participate} Informed consent was obtained from all individual participants included in the study.

%\subsection*{Consent for publication}  The authors have read and understood the publishing policy, and submit this manuscript in accordance with this policy.

%\subsection*{Availability of data and materials} All of the material is owned by the authors and/or no permissions are required.

%\subsection*{Competing interests} We declare that the authors have no competing interests as defined by Springer, or other interests that might be perceived to influence the results and/or discussion reported in this paper.

\begin{itemize}

\item {\bf Funding:} The first author received  financial support from Pontificia Universidad Católica del Perú through project DGI PI0655. 
%VRI-DFI 2016-1-0060.

\item {\bf Data Availability:} Data sharing not applicable to this article as no datasets were generated or analyzed during the current study.

\item {\bf Conflict of interests:} The authors have no competing interests to declare that are relevant to the content of this article.

\end{itemize}

\medskip

\noindent{\bf Acknowledgements:}  
%The authors would like to express their gratitude to  the referee for spotting a gap in the proof of Lemma 3.1.2. 
The first author  thanks  Ralph Gomez and Richard Gonzales for helpful conversations.

\end{document}